\documentclass[seceqn,secthm,9pts]{article}%
\usepackage{amsfonts}
\usepackage{mitpress}
\usepackage{amsmath}
\usepackage{amssymb}
\usepackage{graphicx}
\usepackage{cite}
\usepackage{lineno}%
\providecommand{\U}[1]{\protect\rule{.1in}{.1in}}
\nolinenumbers
\newtheorem{theorem}{Theorem}

\newtheorem{corollary}[theorem]{Corollary}

\newtheorem{definition}[theorem]{Definition}

\newtheorem{remark}[theorem]{Remark}

\newenvironment{proof}[1][Proof]{\noindent \textbf{#1.} }{\  \rule{0.5em}{0.5em}}
\newdimen \dummy
\dummy=\oddsidemargin
\begin{document}

\title{Finite bivariate biorthogonal N - Konhauser polynomials}
\author{Güldoğan Lekesiz, E.*, Çekim, B. and Özarslan, M.A.}
\maketitle
\author{Esra G\"{U}LDO\u{G}AN LEKES\.{I}Z*\\Department of Mathematics, Faculty of Arts and Sciences, \c{C}ankaya University, Ankara 06790, T\"{u}rkiye\\esraguldoganlekesiz@cankaya.edu.tr\\

\and Bayram \c{C}EK\.{I}M\\Department of Mathematics, Faculty of Science, Gazi University, Ankara 06500, T\"{u}rkiye\\bayramcekim@gazi.edu.tr\\

\and Mehmet Ali \"{O}ZARSLAN\\Department of Mathematics, Faculty of Arts and Sciences, Eastern Mediterranean
University, Gazimagusa, TRNC, Mersin 10, T\"{u}rkiye\\mehmetali.ozarslan@emu.edu.tr\\

}

\begin{abstract}
A new set of finite 2D biorthogonal polynomials is defined using the finite
orthogonal polynomials $N_{n}^{\left(  p\right)  }\left(  w\right)  $ and
Konhauser polynomials. We present a connection between this finite 2D
biorthogonal set and the generalized Laguerre-Konhauser polynomials. Also, we
obtain several applications of finite bivariate biorthogonal N - Konhauser polynomials.

\end{abstract}


\section{Introduction}

Several important investigations and studies have been done regarding
univariate and bivariate biorthogonal polynomials \cite{Konhauser2,Konhauser,Carlitz,MT,MT2,AlSalamVerma,MT0}. As is known, the definition of
biorthogonal polynomials is as follows.

Let $u(w)$ and $g(w)$ be real polynomials in $x$ of degree $k>0$ and $l>0$,
respectively. Let $U_{s}(x)$ and$\ G_{n}(x)$ denote polynomials of degree $s$
and $n$ in $u(x)$ and $g(x)$, respectively, then $U_{s}(x)$ and$\ G_{n}(x)$
are polynomials of degree $sk$ and $nl$ in $x$. The polynomials $u(x)$ and
$g(x)$ are called fundamental polynomials.

The real-valued function $\varpi\left(  x\right)  $ of the real variable $x$
is a weight function on the finite or infinite interval $(d_{1},d_{2})$ if all
the moments%
\begin{equation}
I_{i,j}=\int\limits_{d_{1}}^{d_{2}}\varpi\left(  x\right)  \left[  u\left(
x\right)  \right]  ^{i}\left[  g\left(  x\right)  \right]  ^{j}%
dx,\ \ \ i,j=0,1,2,...\label{d}%
\end{equation}
exist, with%
\[
I_{0,0}=\int\limits_{d_{1}}^{d_{2}}\varpi\left(  x\right)  dx\neq0.
\]
From (\ref{d}) it follows that if the integrals $I_{i,j}$ exist for
$i,j=0,1,2,...$, then the integrals%
\[
\int\limits_{d_{1}}^{d_{2}}\varpi\left(  x\right)  x^{i}dx,\ \ \ i=0,1,2,...
\]
exist.

For $s,n\in%
%TCIMACRO{\U{2115} }%
%BeginExpansion
\mathbb{N}
%EndExpansion
_{0}$, if%
\[
J_{s,n}=\int\limits_{d_{1}}^{d_{2}}\varpi\left(  x\right)  U_{s}(x)G_{n}(x)dx=%
%TCIMACRO{\QATOPD{\{}{.}{\ \ \ 0;\ \ s\neq n}{\neq0;\ \ s=n}}%
%BeginExpansion
\genfrac{\{}{.}{0pt}{}{\ \ \ 0;\ \ s\neq n}{\neq0;\ \ s=n}%
%EndExpansion
,
\]
then the sets of the polynomials $\{U_{s}(x)\}$ and $\{G_{n}(x)\}$ are said to
be biorthogonal with respect to the fundamental polynomials $u(x)$ and $g(x)$,
and the weight function $\varpi(x)$ on the interval $(d_{1},d_{2})$
\cite{Konhauser}.\bigskip

\begin{theorem}
Assume that $\varpi(w)$ is a weight function over the interval $(d_{1},d_{2}%
)$, and $u(w)$ and $g(w)$ are basic polynomials corresdponding to the
polynomials $U_{s}(w)$ and $G_{n}(w)$, respectively. For $s,n\in%
%TCIMACRO{\U{2115} }%
%BeginExpansion
\mathbb{N}
%EndExpansion
_{0}$,%
\[
J_{s,n}=\int\limits_{d_{1}}^{d_{2}}\varpi\left(  w\right)  U_{s}(w)G_{n}(w)dw=%
%TCIMACRO{\QATOPD{\{}{.}{\ \ \ 0;\ \ s\neq n}{\neq0;\ \ s=n}}%
%BeginExpansion
\genfrac{\{}{.}{0pt}{}{\ \ \ 0;\ \ s\neq n}{\neq0;\ \ s=n}%
%EndExpansion
\]
so that%
\[
\int\limits_{d_{1}}^{d_{2}}\varpi\left(  w\right)  \left[  u\left(  w\right)
\right]  ^{k}G_{n}(w)dw=%
%TCIMACRO{\QATOPD{\{}{.}{\ \ 0;\ k=0,1,...,n-1}{\neq
%0;\ k=n\ \ \ \ \ \ \ \ \ \ \ \ \ \ \ }}%
%BeginExpansion
\genfrac{\{}{.}{0pt}{}{\ \ 0;\ k=0,1,...,n-1}{\neq
0;\ k=n\ \ \ \ \ \ \ \ \ \ \ \ \ \ \ }%
%EndExpansion
\]
and%
\[
\int\limits_{d_{1}}^{d_{2}}\varpi\left(  w\right)  \left[  g\left(  w\right)
\right]  ^{k}U_{s}(w)dw=%
%TCIMACRO{\QATOPD{\{}{.}{\ \ 0;\ k=0,1,...,s-1}{\neq
%0;\ k=s\ \ \ \ \ \ \ \ \ \ \ \ \ \ \ }}%
%BeginExpansion
\genfrac{\{}{.}{0pt}{}{\ \ 0;\ k=0,1,...,s-1}{\neq
0;\ k=s\ \ \ \ \ \ \ \ \ \ \ \ \ \ \ }%
%EndExpansion
,
\]
are provided \cite{Konhauser2}.
\end{theorem}

In 1967, Konhauser \cite{Konhauser} defined the following biorthogonal pair of
polynomials called by Konhauser polynomials. In following years, Thakare \&
Madhekar also introduced a significant contribution to that field by defining
biorthogonal polynomials suggested by Jacobi, Hermite and Szeg\H{o}-Hermite
polynomials \cite{MT,MT2,MT0}.

For $\chi>-1$ and $\upsilon=1,2,...$, Konhauser defined the biorthogonal pair
of Konhauser polynomials \cite{Konhauser} as%
\begin{equation}
Z_{s}^{\left(  \chi\right)  }\left(  w;\upsilon\right)  =\frac{\Gamma\left(
\upsilon s+\chi+1\right)  }{s!}\sum\limits_{l=0}^{s}\left(  -1\right)
^{l}\binom{s}{l}\frac{w^{\upsilon l}}{\Gamma\left(  \upsilon l+\chi+1\right)
} \label{Zdef}%
\end{equation}
and%
\begin{equation}
Y_{s}^{\left(  \chi\right)  }\left(  w;\upsilon\right)  =\frac{1}{s!}%
\sum\limits_{m=0}^{s}\frac{w^{m}}{m!}\sum\limits_{l=0}^{m}\left(  -1\right)
^{l}\binom{m}{l}\left(  \frac{\chi+l+1}{\upsilon}\right)  _{s}, \label{Ydef}%
\end{equation}
where $\left(  .\right)  _{n}$ and $\Gamma\left(  .\right)  $\ denote the
Pochhammer symbol and the Gamma function, respectively.

He showed that polynomials (\ref{Zdef}) and (\ref{Ydef}) are biorthogonal
polynomials corresponding to $\varpi\left(  w\right)  =w^{\chi}e^{-w}$ over
the interval $\left(  0,\infty\right)  $. That is,%
\begin{equation}
\int\limits_{0}^{\infty}e^{-w}w^{\chi}Z_{s}^{\left(  \chi\right)  }\left(
w;\upsilon\right)  Y_{n}^{\left(  \chi\right)  }\left(  w;\upsilon\right)
dw=\frac{\Gamma\left(  \upsilon s+\chi+1\right)  }{s!}\delta_{n,s} \label{3}%
\end{equation}
is satisfied for $\chi>-1$, $\upsilon=1,2,...$. Here, $\delta_{s,n}$ is the
Kronecker delta, \cite{Konhauser}.

Furthermore, he gave the following recurrence relations for polynomials
$Z_{s}^{\left(  \chi\right)  }\left(  w;\upsilon\right)  $ in \cite{Konhauser}%
:%
\begin{equation}
D_{w}^{\upsilon}\left[  w^{\chi+1}D_{w}Z_{s}^{\left(  \chi\right)  }\left(
w;\upsilon\right)  \right]  =-\upsilon w^{\chi}\left(  \upsilon\left(
s-1\right)  +\chi+1\right)  _{\upsilon}\ Z_{s-1}^{\left(  \chi\right)
}\left(  w;\upsilon\right)  , \label{Zrec1}%
\end{equation}%
\begin{equation}
D_{w}^{\upsilon}\left[  w^{\chi+1}D_{w}Z_{s}^{\left(  \chi\right)  }\left(
w;\upsilon\right)  \right]  =w^{\chi+1}D_{w}Z_{s}^{\left(  \chi\right)
}\left(  w;\upsilon\right)  -s\upsilon w^{\chi}Z_{s}^{\left(  \chi\right)
}\left(  w;\upsilon\right)  \label{Zrec2}%
\end{equation}
and%
\begin{equation}
D_{w}Z_{s}^{\left(  \chi\right)  }\left(  w;\upsilon\right)  =-\upsilon
w^{\upsilon-1}Z_{s-1}^{\left(  \chi+\upsilon\right)  }\left(  w;\upsilon
\right)  , \label{Zrec3}%
\end{equation}
where $D_{w}=\frac{\partial}{\partial w}$.

In the later years, Bin-Saad constructed the generalized Laguerre-Konhauser
polynomials \cite{BinSaad} of form%
\begin{equation}
\ _{\upsilon}L_{s}^{\left(  p,q\right)  }\left(  t,w\right)  =s!\sum
\limits_{k=0}^{s}\sum\limits_{m=0}^{s-k}\frac{\left(  -1\right)  ^{k+m}%
t^{k+p}w^{\upsilon m+q}}{k!m!\left(  s-k-m\right)  !\Gamma\left(
k+p+1\right)  \Gamma\left(  q+\upsilon m+1\right)  } \label{Ldef}%
\end{equation}
for $p,q>-1$ and $\upsilon\in%
%TCIMACRO{\U{2124} }%
%BeginExpansion
\mathbb{Z}
%EndExpansion
^{+}$, or equivalently%
\begin{equation}
\ _{\upsilon}L_{s}^{\left(  p,q\right)  }\left(  t,w\right)  =s!\sum
\limits_{k=0}^{s}\frac{\left(  -1\right)  ^{k}t^{k+p}w^{q}Z_{s-k}^{\left(
q\right)  }\left(  w;\upsilon\right)  }{k!\Gamma\left(  k+p+1\right)
\Gamma\left(  \upsilon\left(  s-k\right)  +q+1\right)  }, \label{LdefZ}%
\end{equation}
where $Z_{s}^{q}\left(  w;\upsilon\right)  $ is first set of the Konhauser
polynomials, defined by (\ref{Zdef}), or%
\begin{equation}
\ _{\upsilon}L_{s}^{\left(  p,q\right)  }\left(  t,w\right)  =s!\sum
\limits_{k=0}^{s}\frac{\left(  -1\right)  ^{k}t^{p}w^{\upsilon k+q}%
L_{s-k}^{\left(  p\right)  }\left(  t\right)  }{k!\Gamma\left(
s-k+p+1\right)  \Gamma\left(  \upsilon k+q+1\right)  }, \label{LdefL}%
\end{equation}
where $L_{s}^{\left(p\right)}\left(  t\right)  $ are the generalized Laguerre polynomials.

Then, \"{O}zarslan and K\"{u}rt \cite{OzKurt} introduced the polynomials%
\[
\ _{\upsilon}%
%TCIMACRO{\tciLaplace}%
%BeginExpansion
\mathcal{L}%
%EndExpansion
_{s}^{\left(  p,q\right)  }\left(  t,w\right)  =L_{s}^{\left(  p\right)
}\left(  t\right)  \sum\limits_{k=0}^{s}Y_{k}^{\left(  q\right)  }\left(
w;\upsilon\right)  ,
\]
forming biorthogonal pair to the generalized Laguerre-Konhauser polynomials,
and it is shown that$\ _{\upsilon}L_{s}^{\left(  p,q\right)  }\left(
t,w\right)  $ and$\ _{\upsilon}%
%TCIMACRO{\tciLaplace}%
%BeginExpansion
\mathcal{L}%
%EndExpansion
_{s}^{\left(  p,q\right)  }\left(  t,w\right)  $ are biorthonormal with
respect to (w.r.t.) $\varpi\left(  t,w\right)  =e^{-\left(  t+w\right)  }$ on
$\left(  0,\infty\right)  \times\left(  0,\infty\right)  $. Also, for
$\upsilon,\chi,q,p$ are complex numbers, they presented a bivariate
Mittag-Leffler functions $E_{p,q,\upsilon}^{(\chi)}\left(  t,w\right)  $,
corresponding to the polynomials $_{\upsilon}L_{s}^{\left(  p,q\right)
}\left(  t,w\right)  $, with%
\begin{equation}
E_{p,q,\upsilon}^{(\chi)}\left(  t,w\right)  =\sum\limits_{m=0}^{\infty}%
\sum\limits_{j=0}^{\infty}\frac{\left(  \chi\right)  _{m+j}t^{m}w^{\upsilon
j}}{m!j!\Gamma\left(  p+m\right)  \Gamma\left(  q+\upsilon j\right)  },
\label{genLKrelE}%
\end{equation}
where real parts of parameters $\chi,p,q$ and $\upsilon\ $are positive, and%
\[
\ _{\upsilon}L_{s}^{\left(  p,q\right)  }\left(  t,w\right)  =t^{p}%
w^{q}E_{p+1,q+1,\upsilon}^{(-s)}\left(  t,w\right)
\]
is given between $_{\upsilon}L_{s}^{\left(  p,q\right)  }\left(  t,w\right)  $
and $E_{p,q,\upsilon}^{(\chi)}\left(  t,w\right)  $, \cite{OzKurt}.

Recently, a method enabling generate bivariate biorthogonal polynomials has
been proved as follows \cite{OzEl}.

\begin{theorem}
Let $U_{s}\left(  w\right)  $ and $G_{s}\left(  w\right)  $ be a pair of
biorthogonal polynomials corresponding to the basic polynomials $u\left(
w\right)  $ and $g\left(  w\right)  $, respectively, w.r.t. the weight
function $\varpi_{2}\left(  w\right)  $ on $\left(  \theta_{1},\theta
_{2}\right)  $. Thus, the biorthogonality condition%
\[
\int\limits_{\theta_{1}}^{\theta_{2}}\varpi_{2}\left(  w\right)  U_{s}\left(
w\right)  G_{n}\left(  w\right)  dw=J_{n,s}:=\left\{
%TCIMACRO{\QATOP{0,\ n\neq s}{J_{n,n},\ \ \ n=s}}%
%BeginExpansion
\genfrac{}{}{0pt}{}{0,\ n\neq s}{J_{n,n},\ \ \ n=s}%
%EndExpansion
\right.
\]
is satisfied. Also, let%
\[
K_{s}\left(  w\right)  =\sum\limits_{i=0}^{s}T_{s,i}\left(  u\left(  w\right)
\right)  ^{i}%
\]
with%
\[
\int\limits_{\theta_{3}}^{\theta_{4}}\varpi_{1}\left(  w\right)  K_{s}\left(
w\right)  K_{n}\left(  w\right)  dw=\left\Vert K_{s}\right\Vert ^{2}%
\delta_{n,s}.
\]
Therefore, the 2D polynomials%
\begin{equation}
P_{s}\left(  t,w\right)  =\sum\limits_{j=0}^{s}\frac{T_{s,j}}{J_{s-j,s-j}%
}\left(  u\left(  t\right)  \right)  ^{j}U_{s-j}\left(  w\right)  \label{PolP}%
\end{equation}
and%
\begin{equation}
F_{s}\left(  t,w\right)  =K_{s}\left(  t\right)  \sum\limits_{k=0}^{s}%
G_{k}\left(  w\right)  \label{PolQ}%
\end{equation}
are biorthogonal w.r.t. the weight function $\varpi_{1}\left(  t\right)
\varpi_{2}\left(  w\right)  $ over $\left(  \theta_{1},\theta_{2}\right)
\times\left(  \theta_{3},\theta_{4}\right)  $ \cite{OzEl}.\bigskip
\end{theorem}

Via this method in \textit{Theorem 2,} the bivariate biorthogonal Hermite
Konhauser polynomials have been defined by \cite{OzEl}%
\begin{equation}
\ _{\upsilon}H_{s}^{\left(  \chi\right)  }\left(  t,w\right)  =\sum
\limits_{k=0}^{\left[  s/2\right]  }\sum\limits_{m=0}^{s-k}\frac{\left(
-1\right)  ^{k}\left(  -s\right)  _{2k}\left(  -s\right)  _{k+m}\left(
2t\right)  ^{s-2k}w^{\upsilon m}}{\left(  -s\right)  _{k}\Gamma\left(
\upsilon m+\chi+1\right)  k!m!},\label{polH}%
\end{equation}
where $\chi>-1$ and $\upsilon=1,2,...$. Further, it has shown that
(\ref{polH}) and%
\[
F_{s}\left(  t,w\right)  =H_{s}\left(  t\right)  \sum\limits_{k=0}^{s}%
Y_{k}^{\left(  \chi\right)  }\left(  w;\upsilon\right)
\]
constitute a pair of bivariate biorthogonal polynomials, where $Y_{k}^{\left(
\chi\right)  }\left(  w;\upsilon\right)  $ and $H_{s}\left(  t\right)  $\ are
the Konhauser (\ref{Ydef}) and the Hermite polynomials, respectively.

Later, \"{O}zarslan and Elidemir \cite[Eq.(8) and Eq.(12)]{OzEl2} derive the
following bivariate Jacobi Konhauser polynomials. The biorthogonal pair of 2D
Jacobi Konhauser polynomials is defined by%
\begin{equation}
\ _{\upsilon}P_{s}^{\left(  p,q\right)  }\left(  t,w\right)  =\frac{\left(
-1\right)  ^{s}\Gamma\left(  s+q+1\right)  }{s!}\sum_{k=0}^{s}\sum_{m=0}%
^{s-k}\frac{\left(  -s\right)  _{k+m}\left(  p+q+1+s\right)  _{k}\left(
1+t\right)  ^{k}w^{\upsilon m}}{2^{k}k!m!\Gamma\left(  q+1+k\right)
\Gamma\left(  q+1+\upsilon m\right)  }\label{JKdef}%
\end{equation}
and%
\begin{equation}
F_{s}\left(  t,w\right)  =P_{s}^{\left(  p,q\right)  }\left(  t\right)
\sum\limits_{k=0}^{s}Y_{k}^{\left(  q\right)  }\left(  w;\upsilon\right)
,\label{JKdefQ}%
\end{equation}
where $\upsilon=1,2,...$ and $p,q>-1$, and $P_{s}^{\left(  p,q\right)
}\left(  t\right)  $ is the classical Jacobi polynomials. The biorthogonality
relation corresponding to polynomials (\ref{JKdef}) and (\ref{JKdefQ}) in
\cite{OzEl2} is%
\begin{align*}
&  \int\limits_{-1}^{1}\int\limits_{0}^{\infty}\left(  1-t\right)  ^{p}\left(
1+t\right)  ^{q}w^{q}\exp\left(  -w\right)  \ _{\upsilon}P_{s}^{\left(
p,q\right)  }\left(  t,w\right)  \ F_{n}\left(  t,w\right)  dwdt\\
&  =\frac{2^{p+q+1+s}\Gamma\left(  q+s+1\right)  \Gamma\left(  s+p+1\right)
}{s!\left(  p+q+1+2s\right)  \Gamma\left(  s+1+p+q\right)  }\delta_{n,s}.
\end{align*}
Also, for $\chi_{1},\chi_{2},\upsilon,p,q\in%
%TCIMACRO{\U{2102} }%
%BeginExpansion
\mathbb{C}
%EndExpansion
$ and $\operatorname{Re}\left(  \chi_{1}\right)  ,\operatorname{Re}\left(
\chi_{2}\right)  ,\operatorname{Re}\left(  \upsilon\right)  ,\operatorname{Re}%
\left(  p\right)  ,\operatorname{Re}\left(  q\right)  >0$, they derive the
bivariate Jacobi Konhauser Mittag-Leffler functions $E_{p,q,\upsilon}^{\left(
\chi_{1};\chi_{2}\right)  }\left(  t,w\right)  $, defined by the formula%
\begin{equation}
E_{p,q,\upsilon}^{\left(  \chi_{1};\chi_{2}\right)  }\left(  t,w\right)
=\sum_{k=0}^{\infty}\sum_{m=0}^{\infty}\frac{\left(  \chi_{1}\right)
_{k+m}\left(  \chi_{2}\right)  _{m}t^{k}w^{\upsilon m}}{\Gamma\left(
p+k\right)  \Gamma\left(  q+\upsilon m\right)  k!m!},\label{Edef}%
\end{equation}
and give the relation between$\ _{\upsilon}P_{s}^{\left(  p,q\right)  }\left(
t,w\right)  $ and $E_{p,q,\upsilon}^{\left(  \chi_{1};\chi_{2}\right)
}\left(  t,w\right)  $ in \cite{OzEl2}.

Then, for the first time, G\"{u}ldo\u{g}an Lekesiz et al. have introduced a set
of finite 2D biorthogonal polynomials, namely the finite 2D biorthogonal I -
Konhauser polynomials \cite{G}\ as follows.

In \cite{G}, for $\upsilon=1,2,...$, polynomials%
\[
\ _{K}I_{s;\upsilon}^{\left(  p,q\right)  }\left(  t,w\right)  =\frac
{\Gamma\left(  p\right)  }{\Gamma\left(  p-s\right)  }\sum_{k=0}^{\left[
s/2\right]  }\sum_{m=0}^{s-k}\frac{\left(  -1\right)  ^{k}\left(  -s\right)
_{2k}\left(  -\left(  s-k\right)  \right)  _{m}\left(  2t\right)
^{s-2k}w^{\upsilon m}}{\left(  -\left(  s-p\right)  \right)  _{k}\Gamma\left(
q+1+\upsilon m\right)  m!k!}%
\]
and%
\[
F_{s}\left(  t,w\right)  :=\ _{K}\mathcal{I}_{s;\upsilon}^{\left(  p,q\right)
}\left(  t,w\right)  =I_{s}^{\left(  p\right)  }\left(  t\right)
\sum\limits_{k=0}^{s}Y_{k}^{\left(  q\right)  }\left(  w;\upsilon\right)  ,
\]
where $I_{s}^{\left(  p\right)  }\left(  t\right)  $ is a finite class of
orthogonal polynomials given in \cite{Masjed}, are finite biorthogonal w.r.t.
$\varpi\left(  t,w\right)  =\left(  1+t^{2}\right)  ^{-\left(  p-1/2\right)
}w^{q}\exp\left(  -w\right)  $ over $\left(  -\infty,\infty\right)
\times\left(  0,\infty\right)  $ and, for $q>-1$ and $s,n=0,1,...,S<p-1$, the
biorthogonality relation is%
\begin{align}
&  \int\limits_{-\infty}^{\infty}\int\limits_{0}^{\infty}\left(
1+t^{2}\right)  ^{-\left(  p-1/2\right)  }w^{q}\exp\left(  -w\right)
\ _{K}I_{s;\upsilon}^{\left(  p,q\right)  }\left(  t,w\right)  \ _{K}%
\mathcal{I}_{n;\upsilon}^{\left(  p,q\right)  }\left(  t,w\right)
dtdw\label{IKort}\\
&  =\frac{s!2^{2p-2}\Gamma^{2}\left(  p\right)  }{\left(  p-1-s\right)
\Gamma\left(  2p-1-s\right)  }\delta_{n,s}.\nonumber
\end{align}

Shortly after, for $p>2S+1$, $S=\max\left\{  s,n\right\}  $, $s,n=0,1,...$ and
$q>-1$, the following finite biorthogonal pair is presented in \cite{G2}%
\[
\ _{K}M_{s;\upsilon}^{\left(  p,q\right)  }\left(  t,w\right)  =\left(
-1\right)  ^{s}\Gamma\left(  q+s+1\right)  \sum_{m=0}^{s}\sum_{r=0}^{s-m}%
\frac{\left(  s+1-p\right)  _{m}\left(  -s\right)  _{m+r}\left(  -t\right)
^{m}w^{\upsilon r}}{\Gamma\left(  q+1+m\right)  \Gamma\left(  q+1+\upsilon
r\right)  m!r!}%
\]
and%
\[
\ _{K}\mathcal{M}_{s;\upsilon}^{\left(  p,q\right)  }\left(  t,w\right)
=M_{s}^{\left(  p,q\right)  }\left(  t\right)  \sum\limits_{k=0}^{s}%
Y_{k}^{\left(  q\right)  }\left(  w;\upsilon\right)  ,
\]
called the finite bivariate biorthogonal M - Konhauser polynomials, are
biorthonormal w.r.t. the weight function $\varpi\left(  t,w\right)
=t^{q}\left(  1+t\right)  ^{-\left(  p+q\right)  }e^{-w}w^{q}$ over the region
$\left(  0,\infty\right)  \times\left(  0,\infty\right)  $. Here,
$M_{s}^{\left(  p,q\right)  }\left(  t\right)  $ is the finite orthogonal
polynomial set given in \cite{Masjed}.

In addition, G\"{u}ldo\u{g}an Lekesiz et al. has shown that the following
relations \cite{G2} holds for the Jacobi Konhauser and the finite M -
Konhauser polynomials:%
\begin{align*}
\ _{K}M_{s;\upsilon}^{\left(  p,q\right)  }\left(  t,w\right)   &
=s!\ _{K}P_{s;\upsilon}^{\left(  -p-q,q\right)  }\left(  -1-2t,w\right)  \\
&  \Leftrightarrow\ _{K}P_{s;\upsilon}^{\left(  p,q\right)  }\left(
t,w\right)  =\frac{1}{s!}\ _{K}M_{s;\upsilon}^{\left(  -p-q,q\right)  }\left(
-\frac{1+t}{2},w\right)
\end{align*}
and%
\begin{align*}
\ _{K}\mathcal{M}_{s;\upsilon}^{\left(  p,q\right)  }\left(  t,w\right)   &
=s!\ F_{s;\upsilon}^{\left(  -p-q,q\right)  }\left(  -1-2t,w\right)  \\
&  \Leftrightarrow F_{s;\upsilon}^{\left(  p,q\right)  }\left(  t,w\right)
=\frac{1}{s!}\ _{K}\mathcal{M}_{s;\upsilon}^{\left(  -p-q,q\right)  }\left(
-\frac{1+t}{2},w\right)  ,
\end{align*}
where $F_{s;\upsilon}^{\left(  p,q\right)  }\left(  t,w\right)  :=F_{s}\left(
t,w\right)  $ is the second set of 2D Jacobi Konhauser polynomials given by
(\ref{JKdefQ}).\bigskip

This study is motivated by papers in \cite{OzKurt, OzEl, G, OzEl2, G2}. Thus,\ we derive another new
set of finite 2D biorthogonal polynomials and their several applications like
fractional calculus operators, Laplace and Fourier transforms,
operational/integral representations and properties like partial differential
equations, generating functions, recurrence relations in the paper. Also, an
important connection is presented between the finite 2D biorthogonal N -
Konhauser polynomials (fNKp) and the generalized Laguerre-Konhauser polynomials.

\bigskip

Before giving definition of fNKp, let's remind the finite univariate
orthogonal polynomials $N_{s}^{\left(  p\right)  }\left(  w\right)  $ (fNp)
satisfying the equation%
\[
w^{2}N_{s}^{\prime\prime}\left(  w\right)  +\left(  \left(  2-p\right)
w+1\right)  N_{s}^{\prime}\left(  w\right)  -s\left(  s+1-p\right)
N_{s}\left(  w\right)  =0,
\]
in \cite{Masjed}. For polynomials $N_{s}^{\left(  p\right)  }\left(  w\right)
$, defined by the serial representation \cite{Masjed}%
\begin{align}
N_{s}^{\left(  p\right)  }\left(  w\right)   &  =\left(  -1\right)  ^{s}%
\sum_{k=0}^{s}k!\binom{p-s-1}{k}\binom{s}{s-k}\left(  -w\right)
^{k}\label{Ndef}\\
&  =w^{s}s!\binom{p-s-1}{s}\ _{1}F_{1}\left(  -s,p-2s;1/w\right)  ,\nonumber
\end{align}
the orthogonality relation%
\begin{equation}
\int\limits_{0}^{\infty}w^{-p}e^{-1/w}N_{s}^{\left(  p\right)  }\left(
w\right)  N_{n}^{\left(  p\right)  }\left(  w\right)  dw=\frac{s!\Gamma\left(
p-s\right)  }{\left(  p-2s-1\right)  }\delta_{n,s} \label{Nort}%
\end{equation}
holds if and only if $p>1$ and $s,n=0,1,...,S<\frac{p-1}{2}$.

\section{The finite bivariate biorthogonal N - Konhauser polynomials}

This section provides us definition of the pair of finite biorthogonal
N-Konhauser polynomial and connections with the generalized Laguerre-Konhauser
polynomials \cite{BinSaad} and the Laguerre-Konhauser polynomials to be
defined in this section.

\begin{definition}
For $\upsilon=1,2,...$, the first set of fNKp is defined by%
\begin{equation}
\ _{K}N_{s;\upsilon}^{\left(  p,q\right)  }\left(  t,w\right)  =\Gamma\left(
p-s\right)  \sum_{k=0}^{s}\sum_{m=0}^{s-k}\frac{\left(  -s\right)
_{k+m}t^{s-k}w^{\upsilon m}}{\Gamma\left(  p-2s+k\right)  \Gamma\left(
q+1+\upsilon m\right)  k!m!} \label{NKdef}%
\end{equation}
where $q>-1$ and $s=0,1,...,S<\frac{p-1}{2}$.
\end{definition}

\begin{remark}
For $w=0$, the first set of fNKp gives fNp, that is%
\[
\ _{K}N_{s;\upsilon}^{\left(  p,q\right)  }\left(  t,0\right)  =\frac
{1}{\Gamma\left(  q+1\right)  }N_{s}^{\left(  p\right)  }\left(  t\right)  ,
\]
or equivalently it gives the generalized Laguerre polynomials as follows:%
\[
\ _{K}N_{s;\upsilon}^{\left(  p,q\right)  }\left(  t,0\right)  =\frac{s!t^{s}%
}{\Gamma\left(  q+1\right)  }L_{s}^{\left(  p-2s-1\right)  }\left(
1/t\right)  .
\]

\end{remark}

\begin{remark}
In particular, by taking $w=0$ and $q=0$,%
\[
\ _{K}N_{s;\upsilon}^{\left(  p,0\right)  }\left(  t,0\right)  =N_{s}^{\left(
p\right)  }\left(  t\right)  ,
\]
or equivalently%
\[
\ _{K}N_{s;\upsilon}^{\left(  p,0\right)  }\left(  t,0\right)  =s!t^{s}%
L_{s}^{\left(  p-2s-1\right)  }\left(  1/t\right)  .
\]

\end{remark}

\begin{remark}
After choosing $p=q=0$ and $\upsilon=1$, we take $w=0$ and get%
\[
\ _{K}N_{s;1}^{\left(  0,0\right)  }\left(  t,0\right)  =L_{s}^{\left(
0\right)  }\left(  t\right)  ,
\]
where $L_{s}^{\left(  0\right)  }\left(  t\right)  $ denotes the generalized
Laguerre polynomials.
\end{remark}

\begin{remark}
For $w=0$, the first set of fNKp can be written in terms of (i.t.o.)
Mittag-Leffler functions defined by Prabhakar \cite{Prab} as%
\[
\ _{K}N_{s;\upsilon}^{\left(  p,q\right)  }\left(  t,0\right)  =\frac
{\Gamma\left(  p-s\right)  t^{s}}{\Gamma\left(  q+1\right)  }E_{1,p-2s}%
^{\left(  -s\right)  }\left(  1/t\right)  .
\]

\end{remark}

\begin{remark}
Considering (\ref{NKdef}) and (\ref{genLKrelE}), the relation%
\[
\ _{K}N_{s;\upsilon}^{\left(  p,q\right)  }\left(  t,w\right)  =t^{s}%
\Gamma\left(  p-s\right)  E_{p-2s,q+1,\upsilon}^{\left(  -s\right)  }\left(
\frac{1}{t},w\right)
\]
holds between polynomials$\ _{K}N_{s;\upsilon}^{\left(  p,q\right)  }\left(
t,w\right)  $ and the bivariate Mittag-Leffler functions corresponding to the
polynomials $\ _{\upsilon}L_{s}^{\left(  p,q\right)  }\left(  t,w\right)  $.
\end{remark}

\begin{remark}
Considering (\ref{NKdef}) and (\ref{Edef}), the first set of fNKp can be
presented i.t.o. Jacobi Konhauser Mittag-Leffler functions of form%
\begin{equation}
\ _{K}N_{s;\upsilon}^{\left(  p,q\right)  }\left(  t,w\right)  =t^{s}%
\Gamma\left(  p-s\right)  E_{p-2s,q+1,\upsilon}^{\left(  -s;-\right)  }\left(
1/t\right)  . \label{NErel}%
\end{equation}

\end{remark}

\begin{theorem}
The first set of fNKp$\ _{K}N_{s;\upsilon}^{\left(  p,q\right)  }\left(
t,w\right)  $ can be expressed as%
\begin{equation}
\ _{K}N_{s;\upsilon}^{\left(  p,q\right)  }\left(  t,w\right)  =s!\Gamma
\left(  p-s\right)  \sum_{k=0}^{s}\frac{\left(  -1\right)  ^{k}t^{s-k}%
Z_{s-k}^{\left(  q\right)  }\left(  w;\upsilon\right)  }{k!\Gamma\left(
p-2s+k\right)  \Gamma\left(  q+1+\upsilon\left(  s-k\right)  \right)  }
\label{NdefZ}%
\end{equation}
i.t.o. the first set of Konhauser polynomials in (\ref{Zdef}).
\end{theorem}

\begin{proof}
Definitions of polynomials $_{K}N_{s;\upsilon}^{\left(  p,q\right)  }\left(
t,w\right)  $ and $Z_{s}^{\left(  \alpha\right)  }\left(  w;\upsilon\right)  $
are directly used.
\end{proof}

\begin{theorem}
The first set of fNKp$\ _{K}N_{s;\upsilon}^{\left(  p,q\right)  }\left(
t,w\right)  $ can be written in the form%
\[
\ _{K}N_{s;\upsilon}^{\left(  p,q\right)  }\left(  t,w\right)  =\Gamma\left(
p-s\right)  \sum_{k=0}^{s}\frac{\left(  -s\right)  _{k}\ t^{s-k}%
w^{-q}\ _{\upsilon}L_{s-k}^{\left(  0,q\right)  }\left(  0,w\right)
}{k!\Gamma\left(  p-2s+k\right)  }%
\]
i.t.o. the generalized Laguerre Konhauser polynomials in (\ref{Ldef}).
\end{theorem}

\begin{proof}
Definitions of polynomials $_{K}N_{s;\upsilon}^{\left(  p,q\right)  }\left(
t,w\right)  $ and $Z_{s}^{\left(  \alpha\right)  }\left(  w;\upsilon\right)  $
are directly used.
\end{proof}

\begin{theorem}
The first set of fNKp$\ _{K}N_{s;\upsilon}^{\left(  p,q\right)  }\left(
t,w\right)  $ can be presented as%
\begin{equation}
\ _{K}N_{s;\upsilon}^{\left(  p,q\right)  }\left(  t,w\right)  =s!t^{s}%
\Gamma\left(  p-s\right)  \sum_{k=0}^{s}\frac{\left(  -1\right)
^{k}w^{\upsilon k}L_{s-k}^{\left(  p-2s-1\right)  }\left(  1/t\right)
}{k!\Gamma\left(  p-s-k\right)  \Gamma\left(  q+1+\upsilon k\right)  }
\label{NdefL}%
\end{equation}
i.t.o. the generalized Laguerre polynomials.
\end{theorem}

\begin{proof}
For the proof definitions of polynomials $_{K}N_{s;\upsilon}^{\left(
p,q\right)  }\left(  t,w\right)  $ and $L_{s}^{\left(  \alpha\right)  }\left(
w\right)  $ are used.
\end{proof}

By \textit{Theorem 2}, the second set of fNKp is as follows%
\begin{equation}
\ _{K}\mathcal{N}_{s;\upsilon}^{\left(  p,q\right)  }\left(  t,w\right)
=N_{s}^{\left(  p\right)  }\left(  t\right)  \sum\limits_{k=0}^{s}%
Y_{k}^{\left(  q\right)  }\left(  w;\upsilon\right)  . \label{Qdef}%
\end{equation}

\begin{corollary}
By inference from Theorem 2, the first set of fNKp $_{K}N_{s;\upsilon
}^{\left(  p,q\right)  }\left(  t,w\right)  $ and the second set of
fNKp$\ _{K}\mathcal{N}_{s;\upsilon}^{\left(  p,q\right)  }\left(  t,w\right)
$ are finite biorthogonal w.r.t. $\varpi\left(  t,w\right)  =t^{-p}w^{q}%
\exp\left(  -w-1/t\right)  $ on region $\left(  0,\infty\right)  \times\left(
0,\infty\right)  $ on condition that $p>2\max\left\{  s\right\}  +1$ and
$q>-1$.
\end{corollary}

\begin{theorem}
For fNKp, the finite biorthogonality relation%
\begin{equation}
\int\limits_{0}^{\infty}\int\limits_{0}^{\infty}t^{-p}w^{q}e^{-w-\frac{1}{t}%
}\ _{K}N_{s;\upsilon}^{\left(  p,q\right)  }\left(  t,w\right)  \ _{K}%
\mathcal{N}_{n;\upsilon}^{\left(  p,q\right)  }\left(  t,w\right)
dwdt=\frac{s!\Gamma\left(  p-s\right)  }{p-2s-1}\delta_{s,n} \label{NKort}%
\end{equation}
holds provided $q>-1$ and$\ s,n=0,1,...,S<\frac{p-1}{2}$.
\end{theorem}

\begin{proof}
Using (\ref{NdefZ}) and (\ref{Qdef}) for the pair of fNKp, we obtain%
\begin{align*}
&  \int\limits_{0}^{\infty}\int\limits_{0}^{\infty}t^{-p}w^{q}e^{-w-\frac
{1}{t}}\ _{K}N_{s;\upsilon}^{\left(  p,q\right)  }\left(  t,w\right)
\ _{K}\mathcal{N}_{n;\upsilon}^{\left(  p,q\right)  }\left(  t,w\right)
dtdw=\int\limits_{0}^{\infty}t^{-p}e^{-1/t}N_{n}^{\left(  p\right)  }\left(
t\right)  \Gamma\left(  p-s\right) \\
&  \times\sum_{k=0}^{s}\frac{\left(  -1\right)  ^{k}s!t^{s-k}}{k!\Gamma\left(
p-2s+k\right)  \Gamma\left(  \upsilon\left(  s-k\right)  +q+1\right)  }%
\int\limits_{0}^{\infty}e^{-w}w^{q}Z_{s-k}^{\left(  q\right)  }\left(
w;\upsilon\right)  \sum_{j=0}^{n}Y_{j}^{\left(  q\right)  }\left(
w;\upsilon\right)  dwdt\\
&  =\int\limits_{0}^{\infty}t^{-p}e^{-1/t}N_{n}^{\left(  p\right)  }\left(
t\right)  \left[  \Gamma\left(  p-s\right)  \sum_{k=0}^{s}\frac{\left(
-s\right)  _{k}\ t^{s-k}}{k!\Gamma\left(  p-2s+k\right)  }\right]  dt.
\end{align*}
Applying (\ref{Ndef}), the desired is obtained from (\ref{Nort}).
\end{proof}

The Kampe de Feriet's double hypergeometric functions \cite{SD} is%
\begin{equation}
F_{r,k,i}^{q,p,n}\left(
%TCIMACRO{\QATOP{\left\{  a_{l}\right\}  _{l=0}^{q}:\left\{  b_{l}\right\}
%_{l=0}^{p};\left\{  t_{l}\right\}  _{l=0}^{n};}{\left\{  d_{l}\right\}
%_{l=0}^{r}:\left\{  f_{l}\right\}  _{l=0}^{k};\left\{  g_{l}\right\}
%_{l=0}^{i};}}%
%BeginExpansion
\genfrac{}{}{0pt}{}{\left\{  a_{j}\right\}  _{j=1}^{q}:\left\{  b_{j}\right\}
_{j=1}^{p};\left\{  t_{j}\right\}  _{j=1}^{n};}{\left\{  d_{j}\right\}
_{j=1}^{r}:\left\{  f_{j}\right\}  _{j=1}^{k};\left\{  g_{j}\right\}
_{j=1}^{i};}%
%EndExpansion
t;w\right)  =\sum\limits_{s,m=0}^{\infty}\frac{\prod\limits_{j=1}^{q}\left(
a_{j}\right)  _{s+m}\prod\limits_{j=1}^{p}\left(  b_{j}\right)  _{s}%
\prod\limits_{j=1}^{n}\left(  t_{j}\right)  _{m}\ t^{m}w^{s}}{\prod
\limits_{j=1}^{r}\left(  d_{j}\right)  _{s+m}\prod\limits_{j=1}^{k}\left(
f_{j}\right)  _{s}\prod\limits_{j=1}^{i}\left(  g_{j}\right)  _{m}\ m!s!}.
\label{KF}%
\end{equation}

\begin{theorem}
The representation of polynomials$\ _{K}N_{s;\upsilon}^{\left(  p,q\right)
}\left(  t,w\right)  $ i.t.o. functions in (\ref{KF}) is%
\begin{equation}
\ _{K}N_{s;\upsilon}^{\left(  p,q\right)  }\left(  t,w\right)  =\frac
{t^{s}\Gamma\left(  p-s\right)  }{\Gamma\left(  p-2s\right)  \Gamma\left(
q+1\right)  }F_{0,1,\upsilon}^{1,0,0}\left(
%TCIMACRO{\QATOP{-s:-;-;}{-:p-2s;\Delta\left(  \upsilon;q+1\right)  ;}}%
%BeginExpansion
\genfrac{}{}{0pt}{}{-s:-;-;}{-:p-2s;\Delta\left(  \upsilon;q+1\right)  ;}%
%EndExpansion
\frac{1}{t};\left(  \frac{w}{\upsilon}\right)  ^{\upsilon}\right)  ,
\label{Nhyprep}%
\end{equation}
where $\Delta\left(  \upsilon;\sigma\right)  =\left\{  \frac{\sigma}{\upsilon
},\frac{\sigma+1}{\upsilon},...,\frac{\sigma+\upsilon-1}{\upsilon}\right\}  $.
\end{theorem}

\begin{proof}
(\ref{Nhyprep}) can be easily obtained by using $\left(  1+\alpha\right)
_{\upsilon r}=\upsilon^{\upsilon r}\prod\limits_{k=0}^{\upsilon-1}\left(
\frac{\alpha+1+k}{\upsilon}\right)  _{r}$ in (\ref{NKdef}).
\end{proof}

\begin{theorem}
The following relation holds between the first set of fNKp and the generalized
Laguerre-Konhauser polynomials defined in \cite{BinSaad}:%
\begin{align}
\ _{K}N_{s;\upsilon}^{\left(  p,q\right)  }\left(  t,w\right)  =\Gamma\left(
p-s\right)  t^{s+p}w^{-q}\ _{\upsilon}L_{s}^{\left(  p-2s-1,q\right)  }\left(
\frac{1}{t},w\right)   & \label{NgenLrel}\\
\Leftrightarrow\ _{\upsilon}L_{s}^{\left(  p,q\right)  }\left(  t,w\right)
=\frac{1}{\Gamma\left(  p+s+1\right)  }t^{p+3s+1}w^{q}\ _{K}N_{s;\upsilon
}^{\left(  p+2s+1,q\right)  }\left(  \frac{1}{t},w\right)  .  & \nonumber
\end{align}

\end{theorem}

\bigskip

The following transition is a well-known relation between fNp and the
classical orthogonal (generalized) Laguerre polynomials \cite{Masjed}:%
\[
N_{s}^{\left(  p\right)  }\left(  w\right)  =s!w^{s}L_{s}^{\left(
p-2s-1\right)  }\left(  \frac{1}{w}\right)  \Leftrightarrow L_{s}^{\left(
p\right)  }\left(  w\right)  =\frac{w^{s}}{s!}N_{s}^{\left(  p+2s+1\right)
}\left(  \frac{1}{w}\right)  .
\]

In this study, we give a similar connection between the fNKp and the
(following) bivariate biorthogonal Laguerre-Konhauser polynomials obtained
from the method mentioned in this study.

If we produce the bivariate biorthogonal Laguerre-Konhauser polynomials with
the help of the method given in \textit{Theorem 2}, then we introduce the
following definition.

\begin{definition}
Applying \textit{Theorem 2}, the first and second set of 2D biorthogonal
Laguerre-Konhauser polynomials that corresponds to the weight function
$\varpi\left(  t,w\right)  =t^{-p}w^{q}\exp\left(  -w-1/t\right)  $ on
$\left(  0,\infty\right)  \times\left(  0,\infty\right)  $ are defined as
follows:%
\begin{align}
\ _{K}L_{s;\upsilon}^{\left(  p,q\right)  }\left(  t,w\right)   &
=\frac{\Gamma\left(  p+1+s\right)  }{s!}\sum_{k=0}^{s}\sum_{m=0}^{s-k}%
\frac{\left(  -s\right)  _{k+m}t^{k}w^{\upsilon m}}{\Gamma\left(
p+1+k\right)  \Gamma\left(  q+1+\upsilon m\right)  k!m!}\label{ourLKdef}\\
&  =\Gamma\left(  s+p+1\right)  \sum_{k=0}^{n}\frac{\left(  -1\right)
^{k}t^{k}Z_{s-k}^{\left(  q\right)  }\left(  w;\upsilon\right)  }%
{\Gamma\left(  p+1+k\right)  \Gamma\left(  q+1+\upsilon\left(  s-k\right)
\right)  k!}\nonumber
\end{align}
and%
\begin{equation}
\ _{K}\mathcal{L}_{s;\upsilon}^{\left(  p,q\right)  }\left(  t,w\right)
=\ _{\upsilon}\mathcal{L}_{s}^{\left(  p,q\right)  }\left(  t,w\right)
=L_{s}^{\left(  p\right)  }\left(  t\right)  \sum_{k=0}^{s}Y_{k}^{\left(
q\right)  }\left(  w;\upsilon\right)  , \label{secLKdef}%
\end{equation}
where$\ _{\upsilon}\mathcal{L}_{s}^{\left(  p,q\right)  }\left(  t,w\right)  $
is given by \cite{OzKurt}.
\end{definition}

\bigskip

Thus, we can analogously present the following relations between pair of fNKp
and pair of$\ _{K}L_{s;\upsilon}^{\left(  p,q\right)  }\left(  t,w\right)  $
and$\ _{K}\mathcal{L}_{s;\upsilon}^{\left(  p,q\right)  }\left(  t,w\right)
$\ defined by (\ref{ourLKdef}):%
\begin{align*}
\ _{K}N_{s;\upsilon}^{\left(  p,q\right)  }\left(  t,w\right)   &
=s!t^{s}\ _{K}L_{s;\upsilon}^{\left(  p-2s-1,q\right)  }\left(  \frac{1}%
{t},w\right) \\
&  \Leftrightarrow\ _{K}L_{s;\upsilon}^{\left(  p,q\right)  }\left(
t,w\right)  =\frac{t^{s}}{s!}\ _{K}N_{s;\upsilon}^{\left(  p+2s+1,q\right)
}\left(  \frac{1}{t},w\right)
\end{align*}
and%
\begin{align*}
\ _{K}\mathcal{N}_{s;\upsilon}^{\left(  p,q\right)  }\left(  t,w\right)   &
=s!t^{s}\ _{K}\mathcal{L}_{s;\upsilon}^{\left(  p-2s-1,q\right)  }\left(
\frac{1}{t},w\right) \\
&  \Leftrightarrow\ _{K}\mathcal{L}_{s;\upsilon}^{\left(  p,q\right)  }\left(
t,w\right)  =\frac{t^{s}}{s!}\ _{K}\mathcal{N}_{s;\upsilon}^{\left(
p+2s+1,q\right)  }\left(  \frac{1}{t},w\right)  .
\end{align*}

\begin{remark}
The bivariate biorthogonal Laguerre-Konhauser polynomials, produced via the
method in \textit{Theorem 2, is unique with the difference of constant factor.
That is, the following relation is satisfied for the }bivariate biorthogonal
Laguerre-Konhauser polynomials$\ _{K}L_{s;\upsilon}^{\left(  p,q\right)
}\left(  t,w\right)  $ and the generalized Laguerre-Konhauser
polynomials$\ _{\upsilon}L_{s}^{\left(  p,q\right)  }\left(  t,w\right)  $:%
\begin{equation}
\ _{\upsilon}L_{s}^{\left(  p,q\right)  }\left(  t,w\right)  =\frac{t^{p}%
w^{q}}{\left(  s+1\right)  _{p}}\ _{K}L_{s;\upsilon}^{\left(  p,q\right)
}\left(  t,w\right)  . \label{twoLKrel}%
\end{equation}

\end{remark}

\begin{corollary}
The first set of fNKp$\ _{K}N_{s;\upsilon}^{\left(  p,q\right)  }\left(
t,w\right)  $ can be written i.t.o. polynomials$\ _{K}L_{s;\upsilon}^{\left(
p,q\right)  }\left(  t,w\right)  $ of form%
\[
\ _{K}N_{s;\upsilon}^{\left(  p,q\right)  }\left(  t,w\right)  =\Gamma\left(
p-s\right)  \sum_{k=0}^{s}\frac{\left(  -s\right)  _{k}\ t^{s-k}}%
{k!\Gamma\left(  p-2s+k\right)  }\ _{K}L_{s-k;\upsilon}^{\left(  0,q\right)
}\left(  0,w\right)  .
\]

\end{corollary}

\begin{proof}
Definitions (\ref{NKdef}) and (\ref{ourLKdef}), for polynomials $_{K}%
N_{s;\upsilon}^{\left(  p,q\right)  }\left(  t,w\right)  $ and$\ _{K}%
L_{s;\upsilon}^{\left(  p,q\right)  }\left(  t,w\right)  $ respectively, are compared.
\end{proof}

\section{Some properties for the fNKp}

In this section, some recurrence relations, generating functions, operational
and integral representations, partial differential equations, Laplace
transform and fractional calculus operators for the first set of fNKp are presented.

\subsection{Generating functions for the first set of fNKp}

In this subsection, we derive some generating functions for polynomials$\ _{K}%
N_{s;\upsilon}^{\left(  p,q\right)  }\left(  t,w\right)  $.

\begin{theorem}
For$\ _{K}N_{s;\upsilon}^{\left(  p,q\right)  }\left(  t,w\right)  $, the
following generating function is satisfied:%
\[
\sum\limits_{s=0}^{\infty}\frac{\Gamma\left(  \beta+s\right)  }{\Gamma\left(
p+s\right)  }\ _{K}N_{s;\upsilon}^{\left(  p+2s,q\right)  }\left(  t,w\right)
\frac{x^{s}}{s!}=\left(  1-tx\right)  ^{-\beta}\ S_{0:1;1}^{1:0;0}\left(
%TCIMACRO{\QATOP{\left[  \beta:1,1\right]  :-;-;}{-:\left[  p:1\right]
%;\left[  q+1;\upsilon\right]  ;}}%
%BeginExpansion
\genfrac{}{}{0pt}{}{\left[  \beta:1,1\right]  :-;-;}{-:\left[  p:1\right]
;\left[  q+1;\upsilon\right]  ;}%
%EndExpansion
\frac{-x}{1-tx},\frac{-txw^{\upsilon}}{1-tx}\right)  ,
\]
where $S_{k:l;m}^{K:L;M}$ is the double hypergeometric series defined by
\cite{SD2}%
\begin{align*}
&  S_{C:D;D^{\prime}}^{A:B;B^{\prime}}\left(
%TCIMACRO{\QATOP{\left[  (a):\theta,\phi\right]  :\left[  (b):\psi\right]
%;\left[  (b^{\prime}):\psi^{\prime}\right]  ;}{\left[  (c):\delta
%,\epsilon\right]  :\left[  (d):\eta\right]  ;\left[  (d^{\prime}):\eta^{\prime
%}\right]  ;}}%
%BeginExpansion
\genfrac{}{}{0pt}{}{\left[  (a):\theta,\phi\right]  :\left[  (b):\psi\right]
;\left[  (b^{\prime}):\psi^{\prime}\right]  ;}{\left[  (c):\delta
,\epsilon\right]  :\left[  (d):\eta\right]  ;\left[  (d^{\prime}):\eta^{\prime
}\right]  ;}%
%EndExpansion
t,w\right) \\
&  =\sum\limits_{s=0}^{\infty}\sum\limits_{n=0}^{\infty}\frac{\prod
\limits_{j=1}^{A}\Gamma\left[  a_{j}+s\theta_{j}+n\phi_{j}\right]
\prod\limits_{j=1}^{B}\Gamma\left[  b_{j}+s\psi_{j}\right]  \prod
\limits_{j=1}^{B^{\prime}}\Gamma\left[  b_{j}^{\prime}+n\psi_{j}^{\prime
}\right]  }{\prod\limits_{j=1}^{C}\Gamma\left[  c_{j}+s\delta_{j}%
+n\epsilon_{j}\right]  \prod\limits_{j=1}^{D}\Gamma\left[  d_{j}+s\eta
_{j}\right]  \prod\limits_{j=1}^{D^{\prime}}\Gamma\left[  d_{j}^{\prime}%
+s\eta_{j}^{\prime}\right]  }\frac{t^{s}}{s!}\frac{w^{n}}{n!}.
\end{align*}

\end{theorem}

\begin{proof}%
\begin{align*}
\sum\limits_{s=0}^{\infty}\frac{\Gamma\left(  \beta+s\right)  }{\Gamma\left(
p+s\right)  }\ _{K}N_{s;\upsilon}^{\left(  p+2s,q\right)  }\left(  t,w\right)
\frac{x^{s}}{s!}=\frac{\Gamma\left(  \beta\right)  }{\Gamma\left(  p\right)
\Gamma\left(  1+q\right)  }\sum\limits_{s=0}^{\infty}\sum_{m=0}^{s}\sum
_{r=0}^{s-m}\frac{\left(  -s\right)  _{m}\left(  -\left(  s-m\right)  \right)
_{r}\left(  \beta\right)  _{s}t^{s-m}w^{\upsilon r}x^{s}}{s!m!r!\left(
p\right)  _{m}\left(  q+1\right)  _{\upsilon r}}  & \\
=\frac{\Gamma\left(  \beta\right)  }{\Gamma\left(  p\right)  \Gamma\left(
1+q\right)  }\sum\limits_{s=0}^{\infty}\sum_{m=0}^{s}\sum_{r=0}^{s-m}%
\frac{\left(  -1\right)  ^{m}\left(  -\left(  s-m\right)  \right)  _{r}\left(
\beta\right)  _{s}t^{s-m}w^{\upsilon r}x^{s}}{\left(  s-m\right)  !m!r!\left(
p\right)  _{m}\left(  q+1\right)  _{\upsilon r}}  & \\
=\frac{\Gamma\left(  \beta\right)  }{\Gamma\left(  p\right)  \Gamma\left(
1+q\right)  }\sum\limits_{s=0}^{\infty}\sum_{m=0}^{\infty}\sum_{r=0}^{s}%
\frac{\left(  -1\right)  ^{m}\left(  -s\right)  _{r}\left(  \beta\right)
_{s+m}t^{s}w^{\upsilon r}x^{s+m}}{s!m!r!\left(  p\right)  _{m}\left(
q+1\right)  _{\upsilon r}}\ \ \ \ \ \ \ \ \ \ \  & \\
=\frac{\Gamma\left(  \beta\right)  }{\Gamma\left(  p\right)  \Gamma\left(
1+q\right)  }\sum\limits_{s=0}^{\infty}\sum_{m=0}^{\infty}\sum_{r=0}^{s}%
\frac{\left(  -1\right)  ^{m+r}\left(  \beta\right)  _{s+m}t^{s}w^{\upsilon
r}x^{s+m}}{\left(  s-r\right)  !m!r!\left(  p\right)  _{m}\left(  q+1\right)
_{\upsilon r}}\ \ \ \ \ \ \ \ \ \ \ \ \ \ \  & \\
=\frac{\Gamma\left(  \beta\right)  }{\Gamma\left(  p\right)  \Gamma\left(
1+q\right)  }\sum\limits_{s=0}^{\infty}\sum_{m=0}^{\infty}\sum_{r=0}^{\infty
}\frac{\left(  \beta\right)  _{s+r+m}\left(  -x\right)  ^{m}\left(
-txw^{\upsilon}\right)  ^{r}\left(  tx\right)  ^{s}}{s!m!r!\left(  p\right)
_{m}\left(  q+1\right)  _{\upsilon r}}\ \ \ \ \ \ \ \ \ \ \ \ \  & \\
=\frac{\Gamma\left(  \beta\right)  }{\Gamma\left(  p\right)  \Gamma\left(
1+q\right)  }\sum_{m=0}^{\infty}\sum_{r=0}^{\infty}\frac{\left(  \beta\right)
_{r+m}\left(  -x\right)  ^{m}\left(  -txw^{\upsilon}\right)  ^{r}}{m!r!\left(
p\right)  _{m}\left(  1+q\right)  _{\upsilon r}}\sum\limits_{s=0}^{\infty
}\left(  \beta+r+m\right)  _{s}\frac{\left(  tx\right)  ^{s}}{n!}  & \\
=\frac{\left(  1-tx\right)  ^{-\beta}\Gamma\left(  \beta\right)  }%
{\Gamma\left(  p\right)  \Gamma\left(  1+q\right)  }\sum_{m=0}^{\infty}%
\sum_{r=0}^{\infty}\frac{\left(  \beta\right)  _{r+m}\left(  \frac{-x}%
{1-tx}\right)  ^{m}\left(  \frac{-txw^{\upsilon}}{1-tx}\right)  ^{r}%
}{m!r!\left(  p\right)  _{m}\left(  1+q\right)  _{\upsilon r}}%
\ \ \ \ \ \ \ \ \ \ \ \ \ \ \ \ \ \ \ \ \ \ \ \  & \\
=\left(  1-tx\right)  ^{-\beta}\sum_{m=0}^{\infty}\sum_{r=0}^{\infty}%
\frac{\Gamma\left(  \beta+r+m\right)  }{m!r!\Gamma\left(  p+m\right)
\Gamma\left(  q+1+\upsilon r\right)  }\left(  \frac{-x}{1-tx}\right)
^{m}\left(  \frac{-txw^{\upsilon}}{1-tx}\right)  ^{r}  & \\
=\left(  1-tx\right)  ^{-\beta}S_{0:1;1}^{1:0;0}\left(
%TCIMACRO{\QATOP{\left[  \beta:1,1\right]  :-;-;}{-:\left[  p:1\right]
%;\left[  q+1;\upsilon\right]  ;}}%
%BeginExpansion
\genfrac{}{}{0pt}{}{\left[  \beta:1,1\right]  :-;-;}{-:\left[  p:1\right]
;\left[  q+1;\upsilon\right]  ;}%
%EndExpansion
\frac{-x}{1-tx},\frac{-txw^{\upsilon}}{1-tx}\right)
.\ \ \ \ \ \ \ \ \ \ \ \ \ \ \ \ \ \  &
\end{align*}

\end{proof}

\begin{theorem}
The first set of fNKp have the generating functions:%
\begin{align}
\sum\limits_{s=0}^{\infty}\frac{1}{s!\left(  p\right)  _{s}}\ _{K}%
N_{s;\upsilon}^{\left(  p+2s,q\right)  }\left(  t,w\right)  \left(  \frac
{x}{t^{3}}\right)  ^{s} &  =\frac{te^{x}}{\Gamma\left(  q+1\right)  }%
\ _{0}F_{1}\left(
%TCIMACRO{\QDATOP{-}{p}}%
%BeginExpansion
\genfrac{}{}{0pt}{0}{-}{p}%
%EndExpansion
;-\frac{x}{t}\right)  \label{gen1}\\
&  \times\ _{0}F_{\upsilon}\left(
%TCIMACRO{\QDATOP{-}{\Delta\left(  \upsilon;q+1\right)  }}%
%BeginExpansion
\genfrac{}{}{0pt}{0}{-}{\Delta\left(  \upsilon;q+1\right)  }%
%EndExpansion
;-x\left(  \frac{w}{\upsilon}\right)  ^{\upsilon}\right)  ,\nonumber
\end{align}%
\begin{align}
\sum\limits_{s=0}^{\infty}\frac{1}{\left(  p\right)  _{s}}\ _{K}N_{s;\upsilon
}^{\left(  p+2s,q\right)  }\left(  t,w\right)  \left(  \frac{x}{t^{3}}\right)
^{s} &  =\frac{t}{\left(  1-x\right)  }\ _{1}F_{1}\left(
%TCIMACRO{\QDATOP{1}{p}}%
%BeginExpansion
\genfrac{}{}{0pt}{0}{1}{p}%
%EndExpansion
;\frac{x}{t\left(  1-x\right)  }\right)  \label{gen2}\\
&  \times\frac{1}{\Gamma\left(  1+q\right)  }\ _{1}F_{\upsilon}\left(
%TCIMACRO{\QDATOP{1}{\Delta\left(  \upsilon;q+1\right)  }}%
%BeginExpansion
\genfrac{}{}{0pt}{0}{1}{\Delta\left(  \upsilon;q+1\right)  }%
%EndExpansion
;\frac{x}{x-1}\left(  \frac{w}{\upsilon}\right)  ^{\upsilon}\right)
,\nonumber
\end{align}%
\begin{align}
&  \sum\limits_{s=0}^{\infty}\frac{\left(  \beta\right)  _{s}}{s!\left(
p\right)  _{s}}\ _{K}N_{s;\upsilon}^{\left(  p+2s,q\right)  }\left(
t,w\right)  \left(  \frac{x}{t^{3}}\right)  ^{s}\label{gen3}\\
&  =\frac{t\left(  1-x\right)  ^{-\beta}}{\Gamma\left(  1+q\right)
}S_{0:1;\upsilon}^{1:0;0}\left(
%TCIMACRO{\QATOP{\beta:-;-;}{-:p;\Delta\left(  \upsilon;1+q\right)  ;}}%
%BeginExpansion
\genfrac{}{}{0pt}{}{\beta:-;-;}{-:p;\Delta\left(  \upsilon;1+q\right)  ;}%
%EndExpansion
\frac{x}{t\left(  x-1\right)  },\frac{x}{x-1}\left(  \frac{w}{\upsilon
}\right)  ^{\upsilon}\right)  \nonumber
\end{align}
and%
\begin{equation}
\sum\limits_{s=0}^{\infty}\frac{t^{p}w^{q}}{\Gamma\left(  p+s\right)  }%
\ _{K}N_{s;\upsilon}^{\left(  p+2s,q\right)  }\left(  \frac{1}{t},w\right)
\frac{\left(  t^{3}x\right)  ^{s}}{s!}=e^{x\left(  1-D_{t}^{-1}-D_{w}%
^{-\upsilon}\right)  }\left\{  \frac{t^{p-1}w^{q}}{\Gamma\left(  p\right)
\Gamma\left(  q+1\right)  }\right\}  .\label{gen4}%
\end{equation}

\end{theorem}

\begin{proof}
For (\ref{gen1}), (\ref{gen2}) and (\ref{gen3}), after using (\ref{NgenLrel}),
we substitute $t\rightarrow\frac{1}{t}$ and $p\rightarrow p-1$ in the
following generating functions for$\ _{\upsilon}L_{s}^{\left(  p,q\right)
}\left(  t,w\right)  $ \cite[eq.(3.3), (3.4) and (3.6)]{BinSaad}%
,\ respectively%
\begin{align*}
\sum\limits_{s=0}^{\infty}\ _{\upsilon}L_{s}^{\left(  p,q\right)  }\left(
t,w\right)  \frac{x^{s}}{s!} &  =\frac{t^{p}w^{q}e^{x}}{\Gamma\left(
1+p\right)  \Gamma\left(  1+q\right)  }\ _{0}F_{1}\left(
%TCIMACRO{\QDATOP{-}{1+p}}%
%BeginExpansion
\genfrac{}{}{0pt}{0}{-}{1+p}%
%EndExpansion
;-tx\right)  \\
&  \times\ _{0}F_{\upsilon}\left(
%TCIMACRO{\QDATOP{-}{\Delta\left(  \upsilon;1+q\right)  }}%
%BeginExpansion
\genfrac{}{}{0pt}{0}{-}{\Delta\left(  \upsilon;1+q\right)  }%
%EndExpansion
;-x\left(  \frac{w}{\upsilon}\right)  ^{\upsilon}\right)  ,
\end{align*}%
\begin{align*}
\sum\limits_{s=0}^{\infty}\ _{\upsilon}L_{s}^{\left(  p,q\right)  }\left(
t,w\right)  x^{s} &  =\frac{1}{\Gamma\left(  p+1\right)  \Gamma\left(
1+q\right)  }\frac{t^{p}w^{q}}{1-x}\ _{1}F_{1}\left(
%TCIMACRO{\QDATOP{1}{1+p}}%
%BeginExpansion
\genfrac{}{}{0pt}{0}{1}{1+p}%
%EndExpansion
;\frac{tx}{1-x}\right)  \\
&  \times\ _{1}F_{\upsilon}\left(
%TCIMACRO{\QDATOP{1}{\Delta\left(  \upsilon;1+q\right)  }}%
%BeginExpansion
\genfrac{}{}{0pt}{0}{1}{\Delta\left(  \upsilon;1+q\right)  }%
%EndExpansion
;\frac{-x}{1-x}\left(  \frac{w}{\upsilon}\right)  ^{\upsilon}\right)
\end{align*}
and%
\begin{align*}
\sum\limits_{s=0}^{\infty}\left(  \beta\right)  _{s}\ _{\upsilon}%
L_{s}^{\left(  p,q\right)  }\left(  t,w\right)  \frac{x^{s}}{s!}=\frac
{t^{p}w^{q}\left(  1-x\right)  ^{-\beta}}{\Gamma\left(  p+1\right)
\Gamma\left(  q+1\right)  }%
\ \ \ \ \ \ \ \ \ \ \ \ \ \ \ \ \ \ \ \ \ \ \ \ \ \ \ \ \ \ \ \ \ \ \ \ \ \
&  \\
\times S_{0:1;\upsilon}^{1:0;0}\left(
%TCIMACRO{\QATOP{\beta:-;-;}{-:p+1;\Delta\left(  \upsilon;q+1\right)  ;}}%
%BeginExpansion
\genfrac{}{}{0pt}{}{\beta:-;-;}{-:p+1;\Delta\left(  \upsilon;q+1\right)  ;}%
%EndExpansion
\frac{-tx}{1-x},\frac{-x}{1-x}\left(  \frac{w}{\upsilon}\right)  ^{\upsilon
}\right)  . &
\end{align*}

In the proof of (\ref{gen4}) we use (\ref{NgenLrel}) and then take
$p\rightarrow p-1$ in the generating function \cite[eq.(3.1)]{BinSaad}%
\[
\sum\limits_{s=0}^{\infty}\ _{\upsilon}L_{s}^{\left(  p,q\right)  }\left(
t,w\right)  \frac{x^{s}}{s!}=e^{x\left(  1-D_{t}^{-1}-D_{w}^{-\upsilon
}\right)  }\left\{  \frac{t^{p}w^{q}}{\Gamma\left(  p+1\right)  \Gamma\left(
q+1\right)  }\right\}
\]
for polynomials$\ _{\upsilon}L_{s}^{\left(  p,q\right)  }\left(  t,w\right)  $.
\end{proof}

\subsection{Some recurrence relations and partial differential equations for
the fNKp}

This subsection includes several recurrence relations and partial differential
equations for the fNKp.

\begin{theorem}
Polynomials$\ _{K}N_{s;\upsilon}^{\left(  p,q\right)  }\left(  t,w\right)  $
have the following recurrence relations:%
\begin{equation}
tD_{t}\left(  \ _{K}N_{s;\upsilon}^{\left(  p,q\right)  }\left(  t,w\right)
\right)  =s\left(  \ _{K}N_{s;\upsilon}^{\left(  p,q\right)  }\left(
t,w\right)  +\ _{K}N_{s-1;\upsilon}^{\left(  p-1,q\right)  }\left(
t,w\right)  \right)  , \label{rec1}%
\end{equation}%
\begin{align}
D_{t}\left(  \ _{K}N_{s;\upsilon}^{\left(  p,q\right)  }\left(  t,w\right)
\right)   &  =s\left(  s+1-p\right)  w^{\upsilon}\left(  \ _{K}N_{s-1;\upsilon
}^{\left(  p-2,q+\upsilon\right)  }\left(  t,w\right)  \right) \label{rec3}\\
&  -s\left(  s-1\right)  \left(  \ _{K}N_{s-2;\upsilon}^{\left(  p-2,q\right)
}\left(  t,w\right)  \right)  +s\left(  p-2s\right)  \left(  \ _{K}%
N_{s-1;\upsilon}^{\left(  p-1,q\right)  }\left(  t,w\right)  \right)
,\nonumber
\end{align}%
\begin{equation}
D_{t}^{m}\left(  t^{s-p+m}\ _{K}N_{s;\upsilon}^{\left(  p,q\right)  }\left(
t,w\right)  \right)  =\frac{\left(  -1\right)  ^{m}\Gamma\left(  p-s\right)
}{\Gamma\left(  p-m-s\right)  }t^{s-p}\ _{K}N_{s;\upsilon}^{\left(
p-m,q\right)  }\left(  t,w\right)  ,\ 0\leq m\leq p-2s-1, \label{rec4}%
\end{equation}%
\begin{equation}
D_{w}^{m}\left(  w^{q}\ _{K}N_{s;\upsilon}^{\left(  p,q\right)  }\left(
t,w\right)  \right)  =w^{q-m}\ _{K}N_{s;\upsilon}^{\left(  p,q-m\right)
}\left(  t,w\right)  ,\ 0\leq m\leq q, \label{rec5}%
\end{equation}%
\begin{equation}
D_{t}^{-m}\left(  t^{s-p-m}\ _{K}N_{s;\upsilon}^{\left(  p,q\right)  }\left(
t,w\right)  \right)  =\frac{\left(  -1\right)  ^{m}\Gamma\left(  p-s\right)
}{\Gamma\left(  p-s+m\right)  }t^{s-p}\ _{K}N_{s;\upsilon}^{\left(
p+m,q\right)  }\left(  t,w\right)  , \label{rec6}%
\end{equation}%
\begin{equation}
D_{w}^{-m}\left(  w^{q}\ _{K}N_{s;\upsilon}^{\left(  p,q\right)  }\left(
t,w\right)  \right)  =w^{q+m}\ _{K}N_{s;\upsilon}^{\left(  p,q+m\right)
}\left(  t,w\right)  , \label{rec7}%
\end{equation}%
\begin{equation}
D_{w}^{\upsilon}\left(  w^{q+1}D_{w}\left(  \ _{K}N_{s;\upsilon}^{\left(
p,q\right)  }\left(  t,w\right)  \right)  \right)  =-s\upsilon\left(
p-s-1\right)  tw^{q}\ _{K}N_{s-1;\upsilon}^{\left(  p-2,q\right)  }\left(
t,w\right)  , \label{rec8}%
\end{equation}%
\begin{align}
D_{w}^{\upsilon}\left(  w^{q+1}D_{w}\left(  \ _{K}N_{s;\upsilon}^{\left(
p,q\right)  }\left(  t,w\right)  \right)  \right)   &  =w^{q+1}D_{w}\left(
\ _{K}N_{s;\upsilon}^{\left(  p,q\right)  }\left(  t,w\right)  \right)
\label{rec9}\\
&  -s\upsilon w^{q}\left(  \ _{K}N_{s;\upsilon}^{\left(  p,q\right)  }\left(
t,w\right)  +\ _{K}N_{s-1;\upsilon}^{\left(  p-1,q\right)  }\left(
t,w\right)  \right)  ,\nonumber
\end{align}%
\begin{equation}
wD_{w}\left(  \ _{K}N_{s;\upsilon}^{\left(  p,q\right)  }\left(  t,w\right)
\right)  =s\upsilon\left(  \ _{K}N_{s;\upsilon}^{\left(  p,q\right)  }\left(
t,w\right)  +\ _{K}N_{s-1;\upsilon}^{\left(  p-1,q\right)  }\left(
t,w\right)  -\left(  p-s-1\right)  t\ _{K}N_{s-1;\upsilon}^{\left(
p-2,q\right)  }\left(  t,w\right)  \right)  , \label{rec10}%
\end{equation}%
\begin{equation}
D_{w}\left(  \ _{K}N_{s;\upsilon}^{\left(  p,q\right)  }\left(  t,w\right)
\right)  =-s\upsilon\left(  p-s-1\right)  tw^{\upsilon-1}\ _{K}N_{s-1;\upsilon
}^{\left(  p-2,q+\upsilon\right)  }\left(  t,w\right)  \label{rec11}%
\end{equation}
for $s\geq0$, and%
\begin{align}
&  \ _{K}N_{s;\upsilon}^{\left(  p,q\right)  }\left(  t,w\right)  +\left(
s-1\right)  t\left(  \ _{K}N_{s-2;\upsilon}^{\left(  p-2,q\right)  }\left(
t,w\right)  \right) \label{rec2}\\
&  =\left(  \left(  p-2s\right)  t-1\right)  \left(  \ _{K}N_{s-1;\upsilon
}^{\left(  p-1,q\right)  }\left(  t,w\right)  \right)  -\left(  p-s-1\right)
tw^{\upsilon}\left(  \ _{K}N_{s-1;\upsilon}^{\left(  p-2,q+\upsilon\right)
}\left(  t,w\right)  \right)  ,\ s\geq2,\nonumber
\end{align}%

\begin{align}
& \left(  p-s-1\right)  t\left(  \ _{K}N_{s-1;\upsilon}^{\left(  p-2,q\right)
}\left(  t,w\right)  -w^{\upsilon}\ _{K}N_{s-1;\upsilon}^{\left(
p-2,q+\upsilon\right)  }\left(  t,w\right)  \right) \label{rec12} \\
& =\ _{K}N_{s;\upsilon}^{\left(  p,q\right)  }\left(  t,w\right)  +\ _{K}N_{s-1;\upsilon}^{\left( p-1,q\right)  }\left(  t,w\right)  ,\ s\geq1, \nonumber
\end{align}
where $ D_{t}^{-1}f\left(  t\right) = \int\limits_{0}^{t}f\left(  x\right)  dx $.
\end{theorem}

\begin{proof}
We differentiate (\ref{NdefL}) w.r.t. variable $t$ and then apply the
recurrence relation \cite{Rainville} $m$ times%
\[
\frac{d^{m}}{dw^{m}}L_{s}^{\left(  p\right)  }\left(  w\right)  =\left\{
%TCIMACRO{\QATOP{\left(  -1\right)  ^{m}L_{s-m}^{\left(  p+m\right)  }\left(
%w\right)  }{0}}%
%BeginExpansion
\genfrac{}{}{0pt}{}{\left(  -1\right)  ^{m}L_{s-m}^{\left(  p+m\right)
}\left(  w\right)  }{0}%
%EndExpansion%
%TCIMACRO{\QATOP{,m\leq s}{,m>s}}%
%BeginExpansion
\genfrac{}{}{0pt}{}{,m\leq s}{,m>s}%
%EndExpansion
\right.
\]
for the generalized Laguerre polynomials. Thus (\ref{rec1}) is obtained.

For (\ref{rec2}), taking $s\rightarrow s+2-k,p\rightarrow p-2s+1$ and
$w\rightarrow1/t$ in the recurrence relation of generalized Laguerre
polynomials \cite{Suetin}%
\[
sL_{s}^{\left(  p\right)  }\left(  w\right)  =\left(  p+1-w\right)
L_{s-1}^{\left(  p+1\right)  }\left(  w\right)  -wL_{s-2}^{\left(  p+2\right)
}\left(  w\right)
\]
definition (\ref{NdefL}) is used.

From (\ref{rec1}) and (\ref{rec2}), we have (\ref{rec3}).

To prove (\ref{rec4}) and (\ref{rec5}), the partial derivatives of definition (\ref{NKdef}) with respect to $t$ and $w$ $m$ times are taken, respectively.

Proofs of (\ref{rec6}) and (\ref{rec7}) can be seen
from (\ref{rec4}) and (\ref{rec5}).

On the other hand, (\ref{rec8}) and (\ref{rec9}) can be proved taking partial
derivative of definition (\ref{NdefZ})\ w.r.t. variable $w$ and then
multiplying with $w^{q+1}$. Finally if the obtained equality is differentiated
$\upsilon$ times w.r.t. variable $w$ and (\ref{Zrec1}), (\ref{Zrec2}) are
used, respectively, the proof is completed.

From (\ref{rec8}) and (\ref{rec9}), we get (\ref{rec10}).

We can easily find (\ref{rec11}) taking partial derivative of (\ref{NdefZ})
w.r.t. variable $w$ and using (\ref{Zrec3}).

Thus, considering (\ref{rec10}) and (\ref{rec11}) together, we have
(\ref{rec12}).
\end{proof}

\begin{theorem}
The first set of fNKp satisfy the following partial differential equations:%
\begin{equation}
\left[  wD_{w}-\upsilon tD_{t}-w^{-q}D_{w}^{\upsilon}w^{q+1}D_{w}\right]
\ _{K}N_{s;\upsilon}^{\left(  p,q\right)  }\left(  t,w\right)  =0 \label{pde2}%
\end{equation}
and%
\begin{equation}
\left[  \upsilon t\left(  s-tD_{t}\right)  \left(  p-s-1-tD_{t}\right)
-w^{-q}D_{w}^{\upsilon}w^{q+1}D_{w}\right]  \ _{K}N_{s;\upsilon}^{\left(
p,q\right)  }\left(  t,w\right)  =0, \label{pde3}%
\end{equation}
where $D_{t}:=\frac{\partial}{\partial t},\ D_{t}^{\upsilon}:=\frac
{\partial^{\upsilon}}{\partial t^{\upsilon}}$ and $D_{w}:=\frac{\partial
}{\partial w},\ D_{w}^{\upsilon}:=\frac{\partial^{\upsilon}}{\partial
w^{\upsilon}}$.
\end{theorem}

\begin{proof}
For (\ref{pde2}), we use (\ref{rec1}) and (\ref{rec9}). On the other hand, (\ref{pde3}) can be proved by (\ref{rec3}) and (\ref{rec11}).
\end{proof}

\begin{theorem}
The following partial differential equation holds for the second set of fNKp:%
\begin{equation}
D_{t}\left[  \left(  1-D_{w}\right)  ^{\upsilon}-1\right]  \ _{K}%
\mathcal{N}_{s;\upsilon}^{\left(  p,q\right)  }\left(  t,w\right)  =s\left(
p-s-1\right)  \left(  1-D_{w}\right)  ^{\upsilon}\ _{K}\mathcal{N}%
_{s-1;\upsilon}^{\left(  p-2,q\right)  }\left(  t,w\right)  . \label{fNKpDE}%
\end{equation}

\end{theorem}

\begin{proof}
By starting the left hand side of (\ref{fNKpDE}), we use%
\[
\left[  \left(  D_{w}-1\right)  ^{\upsilon}-\left(  -1\right)  ^{\upsilon
}\right]  Y_{k+1}^{\left(  q\right)  }\left(  w;\upsilon\right)  =\left(
D_{w}-1\right)  ^{\upsilon}Y_{k}^{\left(  q\right)  }\left(  w;\upsilon
\right), \ \ k=0,1,...,s-1,
\]
given in \cite[Eq.(27)]{Konhauser}\ and%
\[
\frac{d}{dt}N_{s}^{\left(  p\right)  }\left(  t\right)  =s\left(
p-s-1\right)  N_{s-1}^{\left(  p-2\right)  }\left(  t\right)
\]
given by \cite[Eq.(4.20)]{Masjed}, in the definition (\ref{Qdef}).
\end{proof}

\subsection{Operational and integral representations for the first set of
fNKp}

Here, by obtaining operational and integral representations of
polynomials$\ _{K}N_{s;\upsilon}^{\left(  p,q\right)  }\left(  t,w\right)  $
we state them in terms of hypergeometric functions, the generalized
Laguerre-Konhauser polynomials \cite{BinSaad} and finite univariate orthogonal
$N$ polynomials \cite{Masjed}.

\begin{theorem}
For the first set of fNKp$\ _{K}N_{s;\upsilon}^{\left(  p,q\right)  }\left(
t,w\right)  $, the following operational representation%
\begin{equation}
\ _{K}N_{s;\upsilon}^{\left(  p,q\right)  }\left(  t,w\right)  =\frac
{\Gamma\left(  p-s\right)  }{\Gamma\left(  p-2s\right)  }w^{-q}\left(
t\left(  1-D_{w}^{-\upsilon}\right)  \right)  ^{s}\ _{1}F_{1}\left(
-s,p-2s;\frac{1}{t\left(  1-D_{w}^{-\upsilon}\right)  }\right)  \left\{
\frac{w^{q}}{\Gamma\left(  q+1\right)  }\right\}  \label{NKoprel}%
\end{equation}
holds where $_{1}F_{1}$ is generalized hypergeometric function and $D_{w}%
^{-1}$ is defined by%
\[
D_{w}^{-1}f\left(  w\right)  =\int\limits_{0}^{w}f\left(  \xi\right)  d\xi.
\]

\end{theorem}

\begin{proof}%
\begin{align*}
\ _{K}N_{s;\upsilon}^{\left(  p,q\right)  }\left(  t,w\right)   &
=\frac{\Gamma\left(  p-s\right)  t^{s}}{w^{q}}\sum_{k=0}^{s}\frac{\left(
-s\right)  _{k}\left(  \frac{1}{t}\right)  ^{k}}{k!\Gamma\left(
p-2s+k\right)  }\sum_{m=0}^{s-k}\frac{\left(  -\left(  s-k\right)  \right)
_{m}}{m!}D_{w}^{-\upsilon m}\left\{  \frac{w^{q}}{\Gamma\left(  q+1\right)
}\right\} \\
&  =\frac{\Gamma\left(  p-s\right)  }{\Gamma\left(  p-2s\right)  }%
w^{-q}\left(  t\left(  1-D_{w}^{-\upsilon}\right)  \right)  ^{s}\sum_{k=0}%
^{s}\frac{\left(  -s\right)  _{k}\left(  \frac{1}{t\left(  1-D_{w}^{-\upsilon
}\right)  }\right)  ^{k}}{k!\left(  p-2s\right)  _{k}}\left\{  \frac{w^{q}%
}{\Gamma\left(  q+1\right)  }\right\} \\
&  =\frac{\Gamma\left(  p-s\right)  }{\Gamma\left(  p-2s\right)  }%
w^{-q}\left(  t\left(  1-D_{w}^{-\upsilon}\right)  \right)  ^{s}\ _{1}%
F_{1}\left(  -s,p-2s;\frac{1}{t\left(  1-D_{w}^{-\upsilon}\right)  }\right)
\left\{  \frac{w^{q}}{\Gamma\left(  q+1\right)  }\right\}  .
\end{align*}

\end{proof}

\begin{corollary}
The first set of fNKp$\ _{K}N_{s;\upsilon}^{\left(  p,q\right)  }\left(
t,w\right)  $ have the following operational representation i.t.o. fNp and the
generalized Laguerre polynomials $L_{s}^{\left(  p\right)  }\left(  w\right)
$, respectively:%
\[
\ _{K}N_{s;\upsilon}^{\left(  p,q\right)  }\left(  t,w\right)  =w^{-q}%
N_{s}^{\left(  p\right)  }\left(  t\left(  1-D_{w}^{-\upsilon}\right)
\right)  \left\{  \frac{w^{q}}{\Gamma\left(  q+1\right)  }\right\}
\]
and%
\[
\ _{K}N_{s;\upsilon}^{\left(  p,q\right)  }\left(  t,w\right)  =s!w^{-q}%
\left(  t\left(  1-D_{w}^{-\upsilon}\right)  \right)  ^{s}L_{s}^{\left(
p-2s-1\right)  }\left(  \frac{1}{t\left(  1-D_{w}^{-\upsilon}\right)
}\right)  \left\{  \frac{w^{q}}{\Gamma\left(  q+1\right)  }\right\} . 
\]

\end{corollary}

\begin{proof}
The proofs follow from (\ref{Ndef}) and definition of the generalized Laguerre
polynomials i.t.o. hypergeometric functions$\ _{1}F_{1}$.
\end{proof}

\begin{corollary}
For the first set of fNKp$\ _{K}N_{s;\upsilon}^{\left(  p,q\right)  }\left(
t,w\right)  $, the following operational representation holds:%
\[
\ _{K}N_{s;\upsilon}^{\left(  p,q\right)  }\left(  t,w\right)  =\frac
{\Gamma\left(  p-s\right)  }{t^{s+1-p}w^{q}}\left(  1+t^{-2}D_{t}^{-1}%
-D_{w}^{-\upsilon}\right)  ^{s}\left\{  \frac{t^{-\left(  p-2s-1\right)
}w^{q}}{\Gamma\left(  p-2s\right)  \Gamma\left(  q+1\right)  }\right\}  .
\]

\end{corollary}

\begin{proof}%
\begin{align*}
\ _{K}N_{s;\upsilon}^{\left(  p,q\right)  }\left(  \frac{1}{t},w\right)   &
=\frac{\Gamma\left(  p-s\right)  }{t^{p-s-1}w^{q}}\sum_{k=0}^{s}\sum
_{m=0}^{s-k}\frac{\left(  -s\right)  _{k+m}}{k!m!}D_{t}^{-k}\left\{
\frac{t^{p-2s-1}}{\Gamma\left(  p-2s\right)  }\right\}  D_{w}^{-\upsilon
m}\left\{  \frac{w^{q}}{\Gamma\left(  q+1\right)  }\right\} \\
&  =\frac{\Gamma\left(  p-s\right)  }{t^{p-s-1}w^{q}}\left(  1-D_{w}%
^{-\upsilon}\right)  ^{s}\sum_{k=0}^{s}\frac{\left(  -s\right)  _{k}}%
{k!}\left(  \frac{1}{D_{t}\left(  1-D_{w}^{-\upsilon}\right)  }\right)
^{k}\left\{  \frac{t^{p-2s-1}w^{q}}{\Gamma\left(  p-2s\right)  \Gamma\left(
q+1\right)  }\right\} \\
&  =\frac{\Gamma\left(  p-s\right)  }{\Gamma\left(  p-2s\right)  }%
w^{-q}\left(  t\left(  1-D_{w}^{-\upsilon}\right)  \right)  ^{s}\sum_{k=0}%
^{s}\frac{\left(  -s\right)  _{k}\left(  \frac{1}{t\left(  1-D_{w}^{-\upsilon
}\right)  }\right)  ^{k}}{k!\left(  p-2s\right)  _{k}}\left\{  \frac{w^{q}%
}{\Gamma\left(  q+1\right)  }\right\} \\
&  =\Gamma\left(  p-s\right)  t^{-\left(  p-1-s\right)  }w^{-q}\left(
1-D_{t}^{-1}-D_{w}^{-\upsilon}\right)  ^{s}\left\{  \frac{t^{p-2s-1}w^{q}%
}{\Gamma\left(  p-2s\right)  \Gamma\left(  q+1\right)  }\right\}  .
\end{align*}
Thus, substituting $t\rightarrow1/t$ we obtain the desired.
\end{proof}

\begin{theorem}
The following integral representation holds for the first set of fNKp:%
\begin{equation}
\ _{K}N_{s;\upsilon}^{\left(  p,q\right)  }\left(  t,w\right)  =-\frac
{\Gamma\left(  p-s\right)  t^{s}}{4\pi^{2}}\int\limits_{-\infty}^{0^{+}}%
\int\limits_{-\infty}^{0^{+}}e^{y_{1}+y_{2}}y_{1}^{2s-p}y_{2}^{-q-1}\left(
1-\left(  \frac{w}{y_{2}}\right)  ^{\upsilon}-\frac{1}{ty_{1}}\right)
^{s}dy_{1}dy_{2}. \label{NKintrep}%
\end{equation}

\end{theorem}

\begin{proof}
In definition (\ref{NKdef}), using the definition of incomplete Gamma function
given by \cite{Erdelyi,WW}%
\begin{equation}
\frac{1}{\Gamma\left(  w\right)  }=\frac{1}{2\pi i}\int\limits_{-\infty
}^{0^{+}}e^{u}u^{-w}du,\ \ \ \ \left\vert \arg\left(  u\right)  \right\vert
\leq\pi, \label{PGamma}%
\end{equation}
we have (\ref{NKintrep}).
\end{proof}

\begin{theorem}
The first set of fNKp have the following double integral representation:%
\begin{align*}
&  \ _{K}N_{s;\upsilon}^{\left(  p,q\right)  }\left(  t,w\right)
\ _{K}N_{n;\upsilon}^{\left(  \alpha,\beta\right)  }\left(  t,w\right)
=\frac{\Gamma\left(  p-s\right)  \Gamma\left(  \alpha-n\right)  }{16\pi^{4}%
}\int\limits_{-\infty}^{0^{+}}\int\limits_{-\infty}^{0^{+}}\int%
\limits_{-\infty}^{0^{+}}\int\limits_{-\infty}^{0^{+}}e^{y_{1}+y_{2}%
+y_{3}+y_{4}}\\
&  \ \ \ \times y_{1}^{-p}y_{2}^{-\left(  q+1\right)  }y_{3}^{-\alpha}%
y_{4}^{-\left(  \beta+1\right)  }\left(  ty_{1}^{2}\left(  1-\frac
{w^{\upsilon}}{y_{2}^{\upsilon}}\right)  -y_{1}\right)  ^{s}\left(  ty_{3}%
^{2}\left(  1-\frac{w^{\upsilon}}{y_{4}^{\upsilon}}\right)  -y_{3}\right)
^{n}dy_{1}dy_{2}dy_{3}dy_{4}.
\end{align*}

\end{theorem}

\subsection{Laplace transform and fractional calculus operators for the first
set of fNKp}

The Laplace transform is%
\begin{equation}
\mathcal{L}\left\{  f\left(  w\right)  \right\}  \left(  a\right)
=\int\limits_{0}^{\infty}e^{-aw}f\left(  \xi\right)  d\xi,\ \operatorname{Re}%
\left(  a\right)  >0 \label{LaplaceDef}%
\end{equation}
for a one-variable function $f\left(  w\right)  $.

\begin{theorem}
For $\left\vert \frac{y^{\upsilon}}{a^{\upsilon}}\right\vert <1$, the Laplace
transform of the first set of fNKp is obtained as follows:%
\begin{equation}
\mathcal{L}\left\{  w^{q}\ _{K}N_{s;\upsilon}^{\left(  p,q\right)  }\left(
t,yw\right)  \right\}  =\frac{\Gamma\left(  p-s\right)  }{a^{q+1}\Gamma\left(
p-2s\right)  }\left(  \frac{t\left(  a^{\upsilon}-y^{\upsilon}\right)
}{a^{\upsilon}}\right)  ^{s}\ _{1}F_{1}\left[  -s,p-2s;\frac{a^{\upsilon}%
}{t\left(  a^{\upsilon}-y^{\upsilon}\right)  }\right]  . \label{Nlap1D}%
\end{equation}

\end{theorem}

\begin{proof}%
\begin{align*}
&  \mathcal{L}\left\{  w^{q}\ _{K}N_{s;\upsilon}^{\left(  p,q\right)  }\left(
t,yw\right)  \right\}  =\Gamma\left(  p-s\right)  \sum_{k=0}^{s}\sum
_{m=0}^{s-k}\frac{\left(  -s\right)  _{k}\left(  -s+k\right)  _{m}%
t^{s-k}y^{\upsilon m}}{k!m!\Gamma\left(  p-2s+k\right)  \Gamma\left(
q+1+\upsilon m\right)  }\int\limits_{0}^{\infty}e^{-aw}w^{q+\upsilon m}dw\\
&  =\frac{\Gamma\left(  p-s\right)  }{a^{q+1}}\sum_{k=0}^{s}\frac{\left(
-s\right)  _{k}\ t^{s-k}}{k!\Gamma\left(  p-2s+k\right)  }\sum_{m=0}%
^{s-k}\frac{\left(  -s+k\right)  _{m}}{m!}\left(  \frac{y^{\upsilon}%
}{a^{\upsilon}}\right)  ^{m}\\
&  =\frac{\Gamma\left(  p-s\right)  }{a^{q+1}\Gamma\left(  p-2s\right)
}\left(  t\left(  1-\frac{y^{\upsilon}}{a^{\upsilon}}\right)  \right)
^{s}\sum_{k=0}^{s}\frac{\left(  -s\right)  _{k}}{k!\left(  p-2s\right)  _{k}%
}\left(  \frac{a^{\upsilon}}{t\left(  a^{\upsilon}-y^{\upsilon}\right)
}\right)  ^{k}\\
&  =\frac{\Gamma\left(  p-s\right)  }{a^{q+1}\Gamma\left(  p-2s\right)
}\left(  \frac{t\left(  a^{\upsilon}-y^{\upsilon}\right)  }{a^{\upsilon}%
}\right)  ^{s}\ _{1}F_{1}\left[  -s,p-2s;\frac{a^{\upsilon}}{t\left(
a^{\upsilon}-y^{\upsilon}\right)  }\right]  .
\end{align*}

\end{proof}

\begin{corollary}
For the first set of fNKp $_{K}N_{s;\upsilon}^{\left(  p,q\right)  }\left(
t,w\right)  $, Laplace representations i.t.o. the generalized Laguerre
polynomials $L_{s}^{\left(  p\right)  }\left(  w\right)  $ and finite
polynomials $N_{s}^{\left(  p\right)  }\left(  w\right)  $ hold as follows,
respectively:%
\[
\mathcal{L}\left\{  w^{q}\ _{K}N_{s;\upsilon}^{\left(  p,q\right)  }\left(
t,yw\right)  \right\}  =\frac{s!}{a^{q+1}}\left(  \frac{t\left(  a^{\upsilon
}-y^{\upsilon}\right)  }{a^{\upsilon}}\right)  ^{s}L_{s}^{\left(
p-2s-1\right)  }\left(  \frac{a^{\upsilon}}{t\left(  a^{\upsilon}-y^{\upsilon
}\right)  }\right)
\]
and%
\[
\mathcal{L}\left\{  w^{q}\ _{K}N_{s;\upsilon}^{\left(  p,q\right)  }\left(
t,yw\right)  \right\}  =\frac{1}{a^{q+1}}N_{s}^{\left(  p\right)  }\left(
\frac{t\left(  a^{\upsilon}-y^{\upsilon}\right)  }{a^{\upsilon}}\right)  .
\]

\end{corollary}

\begin{proof}
The representations of $N_{s}^{\left(  p\right)  }\left(  w\right)  $ and
$L_{s}^{\left(  p\right)  }\left(  w\right)  $ i.t.o. hypergeometric functions
$_{1}F_{1}$ are used in (\ref{Nlap1D}).
\end{proof}

\bigskip

The 2D Laplace transform in \cite{KG} is%
\[
\mathcal{L}_{2}\left\{  f\left(  t,w\right)  \right\}  \left(  a,b\right)
=\int\limits_{0}^{\infty}\int\limits_{0}^{\infty}e^{-\left(  at+bw\right)
}f\left(  t,w\right)  dtdw,\ \operatorname{Re}\left(  a\right)>0
,\operatorname{Re}\left(  b\right)  >0.
\]

\begin{theorem}
The 2D Laplace transform of the first set of fNKp is given by%
\[
\mathcal{L}_{2}\left\{  t^{p-s-1}w^{q}\ _{K}N_{s;\upsilon}^{\left(
p,q\right)  }\left(  \frac{1}{y_{1}t},y_{2}w\right)  \right\}  =\frac
{\Gamma\left(  p-s\right)  }{a^{p-s}b^{q+1}}\left(  \frac{a}{y_{1}}\left(
1-\frac{y_{2}^{\upsilon}}{b^{\upsilon}}\right)  -1\right)  ^{s}.
\]

\end{theorem}

\begin{proof}
It is computed in a similar way to the proof of \textit{Theorem 30}.
\end{proof}

\bigskip

For $\kappa\in%
%TCIMACRO{\U{2102} }%
%BeginExpansion
\mathbb{C}
%EndExpansion
,\ \operatorname{Re}\left(  \kappa\right)  >0,\ w>a$, the Riemann-Liouville
fractional integral and derivative \cite{Kilbas} are defined by%
\[
\ _{w}\mathbb{I}_{a^{+}}^{\kappa}\left(  f\right)  =\int\limits_{a}^{w}%
\frac{\left(  w-\xi\right)  ^{\kappa-1}}{\Gamma\left(  \kappa\right)
}f\left(  \xi\right)  d\xi,\ \ \ \ f\in L^{1}\left[  a,b\right]
\]
and%
\[
\ _{w}D_{a^{+}}^{\kappa}\left(  f\right)  =\left(  \frac{d}{dw}\right)
^{s}\ _{w}\mathbb{I}_{a^{+}}^{s-\kappa}\left(  f\right)  ,\ \ \ \ f\in
w^{s}\left[  a,b\right]  ,
\]
where $[\operatorname{Re}\left(  \kappa\right)  ]=s-1$ is the integral part of
$\operatorname{Re}\left(  \kappa\right)  $

\begin{theorem}
For $\operatorname{Re}\left(  \tau\right)  \geq0$, $ \tau \in \mathbb{C} $, $w>b$
and $ q+1 > \tau $, the first set of fNKp has the Riemann-Liouville fractional derivative operator as
follows:
\[
\ _{w}D_{b^{+}}^{\tau}\left(  \left(  w-b\right)  ^{q}\ _{K}N_{s;\upsilon
}^{\left(  p,q\right)  }\left(  t,y\left(  w-b\right)  \right)  \right)
=\left(  w-b\right)  ^{q-\tau}\ _{K}N_{s;\upsilon}^{\left(  p,q-\tau\right)
}\left(  t,y\left(  w-b\right)  \right)  .
\]

\end{theorem}

\begin{proof}
For $\operatorname{Re}\left(  \tau\right)  \geq0$,
\begin{align*}
&  \ _{w}D_{b^{+}}^{\tau}\left(  \left(  w-b\right)  ^{q}\ _{K}N_{s;\upsilon
}^{\left(  p,q\right)  }\left(  t,y\left(  w-b\right)  \right)  \right) \\
&  =\frac{\Gamma\left(  p-s\right)  }{\Gamma\left(  s-\tau\right)  }\sum
_{k=0}^{s}\sum_{m=0}^{s-k}\frac{\left(  -s\right)  _{k+m}t^{s-k}y^{\upsilon
m}}{k!m!\Gamma\left(  p-2s+k\right)  \Gamma\left(  q+1+\upsilon m\right)
}D_{w}^{s}\int\limits_{b}^{w}\left(  w-z\right)  ^{s-\tau-1}\left(
z-b\right)  ^{q+\upsilon m}dz\\
&  =\Gamma\left(  p-s\right)  \sum_{k=0}^{s}\sum_{m=0}^{s-k}\frac{\left(
-s\right)  _{k+m}t^{s-k}y^{\upsilon m}}{k!m!\Gamma\left(  p-2s+k\right)
\Gamma\left(  s+q-\tau+1+\upsilon m\right)  }D_{w}^{s}\left(  w-b\right)
^{s+q-\tau+\upsilon m}\\
&  =\left(  w-b\right)  ^{q-\tau}\Gamma\left(  p-s\right)  \sum_{k=0}^{s}%
\sum_{m=0}^{s-k}\frac{\left(  -s\right)  _{k+m}t^{s-k}\left(  y\left(
w-b\right)  \right)  ^{\upsilon m}}{k!m!\Gamma\left(  p-2s+k\right)
\Gamma\left(  q-\tau+1+\upsilon m\right)  }\\
&  =\left(  w-b\right)  ^{q-\tau}\ _{K}N_{s;\upsilon}^{\left(  p,q-\tau
\right)  }\left(  t,y\left(  w-b\right)  \right)  .
\end{align*}

\end{proof}

\begin{theorem}
For the first set of fNKp, we have%
\[
\ _{w}\mathbb{I}_{b^{+}}^{\tau}\left(  \left(  w-b\right)  ^{q}\ _{K}%
N_{s;\upsilon}^{\left(  p,q\right)  }\left(  t,\left(  w-b\right)  y\right)
\right)  =\left(  w-b\right)  ^{q+\tau}\ _{K}N_{s;\upsilon}^{\left(
p,q+\tau\right)  }\left(  t,\left(  w-b\right)  y\right)  ,
\]
where $\operatorname{Re}\left(  \tau\right)  >0$, $ \tau \in \mathbb{C} $ and $w>b$.
\end{theorem}

\begin{proof}%
\begin{align*}
&  _{w}\mathbb{I}_{b^{+}}^{\tau}\left(  \left(  w-b\right)  ^{q}%
\ _{K}N_{s;\upsilon}^{\left(  p,q\right)  }\left(  t,\left(  w-b\right)
y\right)  \right) \\
&  =\frac{\Gamma\left(  p-s\right)  }{\Gamma\left(  \tau\right)  }\sum
_{k=0}^{s}\sum_{m=0}^{s-k}\frac{\left(  -s\right)  _{k+m}t^{s-k}y^{\upsilon
m}}{k!m!\Gamma\left(  p-2s+k\right)  \Gamma\left(  q+1+\upsilon m\right)
}\int\limits_{b}^{w}\left(  w-z\right)  ^{\tau-1}\left(  z-b\right)
^{q+\upsilon m}dz\\
&  =\frac{\Gamma\left(  p-s\right)  }{\Gamma\left(  \tau\right)  }\left(
w-b\right)  ^{q+\tau}\sum_{k=0}^{s}\sum_{m=0}^{s-k}\frac{\left(  -s\right)
_{k+m}t^{s-k}\left(  y\left(  w-b\right)  \right)  ^{\upsilon m}}%
{k!m!\Gamma\left(  p-2s+k\right)  \Gamma\left(  q+1+\upsilon m\right)  }%
\int\limits_{0}^{1}\left(  1-u\right)  ^{\tau-1}u^{q+\upsilon m}du\\
&  =\Gamma\left(  p-s\right)  \left(  w-b\right)  ^{q+\tau}\sum_{k=0}^{s}%
\sum_{m=0}^{s-k}\frac{\left(  -s\right)  _{k+m}t^{s-k}\left(  y\left(
w-b\right)  \right)  ^{\upsilon m}}{k!m!\Gamma\left(  p-2s+k\right)
\Gamma\left(  q+\tau+1+\upsilon m\right)  }\\
&  =\left(  w-b\right)  ^{q+\tau}\ _{K}N_{n;\upsilon}^{\left(  p,q+\tau
\right)  }\left(  t,\left(  w-b\right)  y\right)  ,\ \ \operatorname{Re}%
\left(  \tau\right)  >0.
\end{align*}

\end{proof}

\begin{theorem}
The first set of fNKp have the following double fractional derivative
representation%
\begin{align}
&  \left(  \ _{w}D_{b^{+}}^{\lambda}\ _{t}D_{a^{+}}^{\mu}\right)  \left(
\left(  t-a\right)  ^{p-s-1}\left(  w-b\right)  ^{q}\ _{K}N_{s;\upsilon
}^{\left(  p,q\right)  }\left(  \frac{1}{y_{1}\left(  t-a\right)  }%
,y_{2}\left(  w-b\right)  \right)  \right) \label{Eoprep}\\
&  =\frac{\Gamma\left(  p-s\right)  }{\Gamma\left(  p-\mu-s\right)  }\left(
t-a\right)  ^{p-\mu-s-1}\left(  w-b\right)  ^{q-\lambda}\ _{K}N_{s;\upsilon
}^{\left(  p-\mu,q-\lambda\right)  }\left(  \frac{1}{y_{1}\left(  t-a\right)
},y_{2}\left(  w-b\right)  \right)  \nonumber
\end{align}
for $ \lambda, \mu \in \mathbb{C} $, $\operatorname{Re}\left(  \lambda\right)  >0, \operatorname{Re}\left(  \mu\right)  >0$,  $w>b$,  $t>a$ and $q+1>\lambda, 
\ p>\mu+2S+1$, \ where $S=1,2,...,s$.
\end{theorem}

\begin{proof}
(\ref{Eoprep}) is proved similar to the proof of \textit{Theorem 33}.
\end{proof}

\begin{theorem}
For the first set of fNKp, we have%
\begin{align*}
&  \left(  \ _{w}I_{b^{+}}^{\lambda}\ _{t}I_{a^{+}}^{\mu}\right)  \left(
\left(  t-a\right)  ^{p-s-1}\left(  w-b\right)  ^{q}\ _{K}N_{s;\upsilon
}^{\left(  p,q\right)  }\left(  \frac{1}{y_{1}\left(  t-a\right)  }%
,y_{2}\left(  w-b\right)  \right)  \right) \\
&  =\frac{\Gamma\left(  p-s\right)  }{\Gamma\left(  p+\mu-s\right)  }\left(
t-a\right)  ^{p+\mu-s-1}\left(  w-b\right)  ^{q+\lambda}\ _{K}N_{s;\upsilon
}^{\left(  p+\mu,q+\lambda\right)  }\left(  \frac{1}{y_{1}\left(  t-a\right)
},y_{2}\left(  w-b\right)  \right)  
\end{align*}
for $ \lambda, \mu \in \mathbb{C} $, $\operatorname{Re}\left(  \lambda\right)  >0, \operatorname{Re}\left(  \mu\right)  >0$ and  $w>b$,  $t>a$.

\end{theorem}

\begin{proof}
Inspired by the proof of \textit{Theorem 34}, this proof is followed.
\end{proof}

\section{A new class of finite 2D biorthogonal functions derived from fKNp}

For a bivariate function $d(t,w)$, the Fourier transform in \cite{Davies} is
as%
\begin{equation}%
%TCIMACRO{\tciFourier}%
%BeginExpansion
\mathcal{F}%
%EndExpansion
\left(  d\left(  t,w\right)  \right)  =\int\limits_{-\infty}^{\infty}%
\int\limits_{-\infty}^{\infty}\exp\left(  -i\xi_{1}t-i\xi_{2}w\right)
d\left(  t,w\right)  dtdw \label{Fourier}%
\end{equation}
and Parseval identity corresponding to (\ref{Fourier}) is defined of form%
\begin{equation}
\int\limits_{-\infty}^{\infty}\int\limits_{-\infty}^{\infty}d\left(
t,w\right)  \overline{f\left(  t,w\right)  }dtdw=\frac{1}{\left(  2\pi\right)
^{2}}\int\limits_{-\infty}^{\infty}\int\limits_{-\infty}^{\infty}%
%TCIMACRO{\tciFourier}%
%BeginExpansion
\mathcal{F}%
%EndExpansion
\left(  d\left(  t,w\right)  \right)  \overline{%
%TCIMACRO{\tciFourier}%
%BeginExpansion
\mathcal{F}%
%EndExpansion
\left(  f\left(  t,w\right)  \right)  }d\xi_{2}d\xi_{1},\ d,f\in L^{2}\left(
%TCIMACRO{\U{211d} }%
%BeginExpansion
\mathbb{R}
%EndExpansion
\right)  . \label{Parseval}%
\end{equation}

Let us consider the following functions%
\[%
%TCIMACRO{\QATOPD{\{}{.}{d\left(  t,w\right)  =\exp\left(  -p_{1}t+q_{1}%
%w-\frac{e^{w}+e^{-t}}{2}\right)  \ _{K}N_{s;\upsilon}^{\left(  \alpha
%,\beta\right)  }\left(  e^{t},e^{w}\right)  }{f\left(  t,w\right)
%=\exp\left(  -p_{2}t+q_{2}w-\frac{e^{w}+e^{-t}}{2}\right)  \ _{K}%
%\QTR{cal}{N}_{n;\upsilon}^{\left(  \gamma,\delta\right)  }\left(  e^{t}%
%,e^{w}\right)  }}%
%BeginExpansion
\genfrac{\{}{.}{0pt}{}{d\left(  t,w\right)  =\exp\left(  -p_{1}t+q_{1}%
w-\frac{e^{w}+e^{-t}}{2}\right)  \ _{K}N_{s;\upsilon}^{\left(  \alpha
,\beta\right)  }\left(  e^{t},e^{w}\right)  }{f\left(  t,w\right)
=\exp\left(  -p_{2}t+q_{2}w-\frac{e^{w}+e^{-t}}{2}\right)  \ _{K}%
\mathcal{N}_{n;\upsilon}^{\left(  \gamma,\delta\right)  }\left(  e^{t}%
,e^{w}\right)  }%
%EndExpansion
\]
and calculate the corresponding Fourier transform, using the transforms $\exp
t=u$, $\exp w=z$, and the definition (\ref{NKdef}), as follows%
\begin{align*}
&
%TCIMACRO{\tciFourier}%
%BeginExpansion
\mathcal{F}%
%EndExpansion
\left(  d\left(  t,w\right)  \right)  =\int\limits_{-\infty}^{\infty}%
\int\limits_{-\infty}^{\infty}e^{-i\left(  \xi_{1}t+\xi_{2}w\right)  }d\left(
t,w\right)  dtdw\\
&  =\int\limits_{-\infty}^{\infty}\int\limits_{-\infty}^{\infty}e^{-i\left(
\xi_{1}t+\xi_{2}w\right)  }e^{-p_{1}t+q_{1}w}\exp\left(  -\frac{e^{w}+e^{-t}%
}{2}\right)  \ _{K}N_{s;\upsilon}^{\left(  \alpha,\beta\right)  }\left(
e^{t},e^{w}\right)  dtdw\\
&  =\sum_{k=0}^{s}\sum_{m=0}^{s-k}\frac{\Gamma\left(  \alpha-s\right)  \left(
-s\right)  _{k+m}}{k!m!\Gamma\left(  \alpha-2s+k\right)  \Gamma\left(
\beta+1+\upsilon m\right)  }\int\limits_{0}^{\infty}u^{s-p_{1}-i\xi_{1}%
-1-k}e^{-\frac{1}{2u}}du\int\limits_{0}^{\infty}e^{-\frac{z}{2}}z^{q_{1}%
-i\xi_{2}-1+\upsilon m}dz.
\end{align*}
With the help of the definitions of the Beta integral and the Gamma integral
\cite{Erdelyi,WW} and the transforms $1/y=2u$ and $z/2=x$, we obtain%
\begin{align}
&
%TCIMACRO{\tciFourier}%
%BeginExpansion
\mathcal{F}%
%EndExpansion
\left(  d\left(  t,w\right)  \right)  =-2^{-s+p_{1}+q_{1}+i\xi_{1}-i\xi_{2}%
}\Gamma\left(  \alpha-s\right)  \sum_{k=0}^{s}\sum_{m=0}^{s-k}\frac{\left(
-s\right)  _{k+m}2^{k}2^{\upsilon m}}{k!m!\Gamma\left(  \alpha-2s+k\right)
\Gamma\left(  \beta+1+\upsilon m\right)  }\label{Fourier1}\\
&  \times\int\limits_{0}^{\infty}y^{-s+k+p_{1}+i\xi_{1}-1}e^{-y}%
dy\int\limits_{0}^{\infty}x^{q_{1}+\upsilon m-i\xi_{2}-1}e^{-x}dx\nonumber\\
&  =\frac{-2^{-s+p_{1}+q_{1}+i\xi_{1}-i\xi_{2}}\Gamma\left(  \alpha-s\right)
}{\Gamma\left(  \alpha-2s\right)  \Gamma\left(  \beta+1\right)  }\Gamma\left(
p_{1}-s+i\xi_{1}\right)  \Gamma\left(  q_{1}-i\xi_{2}\right)  \sum_{k=0}%
^{s}\sum_{m=0}^{s-k}\frac{\left(  -s\right)  _{k+m}\left(  p_{1}-s+i\xi
_{1}\right)  _{k}\left(  q_{1}-i\xi_{2}\right)  _{\upsilon m}}{k!m!\left(
\alpha-2s\right)  _{k}\left(  \beta+1\right)  _{\upsilon m}},\nonumber
\end{align}
and similarly we have%
\begin{align}
&
%TCIMACRO{\tciFourier}%
%BeginExpansion
\mathcal{F}%
%EndExpansion
\left(  f\left(  t,w\right)  \right)  =\int\limits_{-\infty}^{\infty}%
\int\limits_{-\infty}^{\infty}e^{-i\left(  \xi_{1}t+\xi_{2}w\right)  }f\left(
t,w\right)  dtdw\label{Fourier2}\\
&  =\int\limits_{-\infty}^{\infty}\int\limits_{-\infty}^{\infty}e^{-i\left(
\xi_{1}t+\xi_{2}w\right)  }e^{-p_{2}t+q_{2}w}\exp\left(  -\frac{e^{w}+e^{-t}%
}{2}\right)  \ _{K}\mathcal{N}_{n;\upsilon}^{\left(  \gamma,\delta\right)
}\left(  e^{t},e^{w}\right)  dtdw\nonumber\\
&  =\sum_{k=0}^{n}\frac{\Gamma\left(  \gamma-n\right)  \left(  -n\right)
_{k}}{k!\Gamma\left(  \gamma-2n+k\right)  }\int\limits_{0}^{\infty
}u^{n-k-p_{2}-i\xi_{1}-1}e^{-\frac{1}{2u}}du\sum_{r=0}^{n}\frac{1}{r!}%
\sum_{j=0}^{r}\frac{1}{j!}\sum_{l=0}^{j}\left(  -1\right)  ^{l}\binom{j}%
{l}\left(  \frac{l+\delta+1}{\upsilon}\right)  _{r}\int\limits_{0}^{\infty
}e^{-\frac{z}{2}}z^{q_{2}+j-i\xi_{2}-1}dz\nonumber\\
&  =2^{p_{2}+q_{2}-n+i\xi_{1}-i\xi_{2}}\frac{\Gamma\left(  \gamma-n\right)
}{\Gamma\left(  \gamma-2n\right)  }\Gamma\left(  p_{2}-n+i\xi_{1}\right)
\Gamma\left(  q_{2}-i\xi_{2}\right)  \ _{2}F_{1}\left(
%TCIMACRO{\QATOP{-n,p_{2}-n+i\xi_{1}}{\gamma-2n}}%
%BeginExpansion
\genfrac{}{}{0pt}{}{-n,p_{2}-n+i\xi_{1}}{\gamma-2n}%
%EndExpansion
;2\right) \nonumber\\
&  \times\sum_{m=0}^{n}\frac{1}{m!}\sum_{r=0}^{m}\frac{\left(  q_{2}-i\xi
_{2}\right)  _{r}\ 2^{r}}{r!}\sum_{l=0}^{r}\left(  -1\right)  ^{l}\binom{r}%
{l}\left(  \frac{l+1+\delta}{\upsilon}\right)  _{m}.\nonumber
\end{align}
So the Fourier transform (\ref{Fourier1}) becomes%
\begin{equation}%
%TCIMACRO{\tciFourier}%
%BeginExpansion
\mathcal{F}%
%EndExpansion
\left(  d\left(  t,w\right)  \right)  =\frac{\Gamma\left(  \alpha-s\right)
}{\Gamma\left(  \beta+1\right)  \Gamma\left(  \alpha-2s\right)  }G_{1}\left(
s,p_{1},q_{1};\xi_{1},\xi_{2}\right)  \Psi_{1}\left(  s,p_{1},q_{1}%
,\alpha,\beta,\upsilon;\xi_{1},\xi_{2}\right)  , \label{F1}%
\end{equation}
where%
\[
G_{1}\left(  s,p_{1},q_{1};\xi_{1},\xi_{2}\right)  =\frac{\left(  -1\right)
^{s+1}2^{p_{1}+q_{1}-s+i\xi_{1}-i\xi_{2}}}{\left(  1-p_{1}-i\xi_{1}\right)
_{s}}\Gamma\left(  p_{1}+i\xi_{1}\right)  \Gamma\left(  q_{1}-i\xi_{2}\right)
\]
and%
\[
\Psi_{1}\left(  s,p_{1},q_{1},\alpha,\beta,\upsilon;\xi_{1},\xi_{2}\right)
=\sum_{k=0}^{s}\sum_{m=0}^{s-k}\frac{\left(  -s\right)  _{k+m}\left(
p_{1}-s+i\xi_{1}\right)  _{k}\left(  q_{1}-i\xi_{2}\right)  _{\upsilon m}%
}{k!m!\left(  \alpha-2s\right)  _{k}\left(  \beta+1\right)  _{\upsilon m}}.
\]
On the other hand, $%
%TCIMACRO{\tciFourier}%
%BeginExpansion
\mathcal{F}%
%EndExpansion
\left(  f\left(  t,w\right)  \right)  $ can be calculated as%
\begin{equation}%
%TCIMACRO{\tciFourier}%
%BeginExpansion
\mathcal{F}%
%EndExpansion
\left(  f\left(  t,w\right)  \right)  =\frac{\Gamma\left(  \gamma-n\right)
}{\Gamma\left(  \gamma-2n\right)  }G_{2}\left(  n,p_{2},q_{2};\xi_{1},\xi
_{2}\right)  \Psi_{2}\left(  n,p_{2},q_{2},\gamma,\delta,\upsilon;\xi_{1}%
,\xi_{2}\right)  , \label{F2}%
\end{equation}
where%
\[
G_{2}\left(  n,p_{2},q_{2};\xi_{1},\xi_{2}\right)  =\frac{\left(  -1\right)
^{n}2^{p_{2}+q_{2}-n+i\xi_{1}-i\xi_{2}}}{\left(  1-p_{2}-i\xi_{1}\right)
_{n}}\Gamma\left(  p_{2}+i\xi_{1}\right)  \Gamma\left(  q_{2}-i\xi_{2}\right)
\]
and%
\begin{align*}
\Psi_{2}\left(  n,p_{2},q_{2},\gamma,\delta,\upsilon;\xi_{1},\xi_{2}\right)
&  =\ _{2}F_{1}\left(
%TCIMACRO{\QATOP{-n,p_{2}-n+i\xi_{1}}{\gamma-2n}}%
%BeginExpansion
\genfrac{}{}{0pt}{}{-n,p_{2}-n+i\xi_{1}}{\gamma-2n}%
%EndExpansion
;2\right)  \sum_{m=0}^{n}\frac{1}{m!}\sum_{r=0}^{m}\frac{\left(  q_{2}%
-i\xi_{2}\right)  _{r}\ 2^{r}}{r!}\\
&  \times\sum_{l=0}^{r}\left(  -1\right)  ^{l}\binom{r}{l}\left(
\frac{l+1+\delta}{\upsilon}\right)  _{m}.
\end{align*}
In Parseval's identity, using (\ref{Fourier1}) and (\ref{Fourier2}) and
replacing $e^{t}=u$ and $e^{w}=z$, the equality%
\begin{align}
&  \int\limits_{0}^{\infty}\int\limits_{0}^{\infty}u^{-\left(  p_{1}%
+p_{2}+1\right)  }z^{q_{1}+q_{2}-1}e^{-z-1/u}\ _{K}N_{s;\upsilon}^{\left(
\alpha,\beta\right)  }\left(  u,z\right)  \ _{K}\mathcal{N}_{n;\upsilon
}^{\left(  \gamma,\delta\right)  }\left(  u,z\right)  dudz\label{Eq}\\
&  =\frac{\Gamma\left(  \alpha-s\right)  \Gamma\left(  \gamma-n\right)
}{\left(  2\pi\right)  ^{2}\Gamma\left(  \beta+1\right)  \Gamma\left(
\alpha-2s\right)  \Gamma\left(  \gamma-2n\right)  }\int\limits_{-\infty
}^{\infty}\int\limits_{-\infty}^{\infty}G_{1}\left(  s,p_{1},q_{1};\xi_{1}%
,\xi_{2}\right)  \overline{G_{2}\left(  n,p_{2},q_{2};\xi_{1},\xi_{2}\right)
}\nonumber\\
&  \times\Psi_{1}\left(  s,p_{1},q_{1},\alpha,\beta,\upsilon;\xi_{1},\xi
_{2}\right)  \overline{\Psi_{2}\left(  n,p_{2},q_{2},\gamma,\delta
,\upsilon;\xi_{1},\xi_{2}\right)  }d\xi_{1}d\xi_{2}\nonumber
\end{align}
is obtained. Thus, if $p_{1}+p_{2}+1=\alpha=\gamma$ and $q_{1}+q_{2}%
-1=\beta=\delta$ are chosen in the left-hand side of (\ref{Eq}) then ,applying
the finite biorthogonality relation for fKNp, the equation (\ref{Eq}) reads as%
\begin{align*}
&  \int\limits_{0}^{\infty}\int\limits_{0}^{\infty}u^{-\left(  p_{1}%
+p_{2}+1\right)  }z^{q_{1}+q_{2}-1}\exp\left(  -z-1/u\right)  \ _{K}%
N_{s;\upsilon}^{\left(  p_{1}+p_{2}+1,q_{1}+q_{2}-1\right)  }\left(
u,z\right) \\
&  \times\ _{K}\mathcal{N}_{n;\upsilon}^{\left(  p_{2}+p_{1}+1,q_{2}%
+q_{1}-1\right)  }\left(  u,z\right)  dudz=\frac{\Gamma\left(  p_{1}%
+p_{2}-s+1\right)  s!}{p_{1}+p_{2}-2s}\delta_{s,n}\\
&  =\frac{\Gamma\left(  \alpha-s\right)  \Gamma\left(  \gamma-n\right)
}{\left(  2\pi\right)  ^{2}\Gamma\left(  \alpha-2s\right)  \Gamma\left(
\gamma-2n\right)  \Gamma\left(  \beta+1\right)  }\int\limits_{-\infty}%
^{\infty}\int\limits_{-\infty}^{\infty}G_{1}\left(  s,p_{1},q_{1};\xi_{1}%
,\xi_{2}\right)  \overline{G_{2}\left(  n,p_{2},q_{2};\xi_{1},\xi_{2}\right)
}\\
&  \times\overline{\Psi_{2}\left(  n,p_{2},q_{2},p_{2}+p_{1}+1,q_{2}%
+q_{1}-1,\upsilon;\xi_{1},\xi_{2}\right)  }\\
&  \times\Psi_{1}\left(  s,p_{1},q_{1},p_{1}+p_{2}+1,q_{1}+q_{2}%
-1,\upsilon;\xi_{1},\xi_{2}\right)  d\xi_{1}d\xi_{2}.
\end{align*}
That is,%
\begin{align*}
&  \int\limits_{-\infty}^{\infty}\int\limits_{-\infty}^{\infty}G_{1}\left(
n,p_{1},q_{1};\xi_{1},\xi_{2}\right)  \Psi_{1}\left(  n,p_{1},q_{1}%
,p_{1}+p_{2}+1,q_{1}+q_{2}-1,\upsilon;\xi_{1},\xi_{2}\right) \\
&  \times\overline{G_{2}\left(  s,p_{2},q_{2};\xi_{1},\xi_{2}\right)
}\overline{\Psi_{2}\left(  s,p_{2},q_{2},p_{2}+p_{1}+1,q_{2}+q_{1}%
-1,\upsilon;\xi_{1},\xi_{2}\right)  }d\xi_{1}d\xi_{2}\\
&  =\frac{\left(  2\pi\right)  ^{2}s!\Gamma^{2}\left(  p_{1}+p_{2}%
+1-2s\right)  \Gamma\left(  q_{1}+q_{2}\right)  }{\left(  p_{1}+p_{2}%
-2s\right)  \Gamma\left(  p_{1}+p_{2}+1-s\right)  }\delta_{n,s}.
\end{align*}

Then, we arrive the following theorem.

\begin{theorem}
For $q_{1},q_{2}>0$ and $p_{1},p_{2}>S=\max\left\{  s,n\right\}  $, where $p_{1},p_{2} \notin \mathbb{Z}$,
the families of functions$\ \\ \Upsilon_{1}\left(  s,p_{1},q_{1},q_{2}%
,p_{2},\upsilon;it,iw\right)  $ and $\Upsilon_{2}\left(  n,p_{2},q_{2}%
,q_{1},p_{1},\upsilon;it,iw\right)  $ are finite biorthogonal functions \\
corresponding to the weight function $\varpi\left(  p_{1},p_{2},q_{1}%
,q_{2};it,iw\right)  $\ and the finite biorthogonality relation is%
\begin{align*}
&  \int\limits_{-\infty}^{\infty}\int\limits_{-\infty}^{\infty}\varpi\left(
p_{1},p_{2},q_{1},q_{2};it,iw\right)  \Upsilon_{1}\left(  s,p_{1},q_{1}%
,q_{2},p_{2},\upsilon;it,iw\right)  \Upsilon_{2}\left(  n,p_{2},q_{2}%
,q_{1},p_{1},\upsilon;-it,-iw\right)  dwdt\\
&  =\frac{\left(  2\pi\right)  ^{2}s!\Gamma^{2}\left(  p_{1}+p_{2}%
+1-2s\right)  \Gamma\left(  q_{1}+q_{2}\right)  }{\left(  p_{1}+p_{2}%
-2s\right)  \Gamma\left(  p_{1}+p_{2}+1-s\right)  }\delta_{s,n},
\end{align*}
where the corresponding weight function is%
\[
\varpi\left(  p_{1},p_{2},q_{1},q_{2};it,iw\right)  =\Gamma\left(
p_{1}+it\right)  \Gamma\left(  p_{2}-it\right)  \Gamma\left(  q_{1}-iw\right)
\Gamma\left(  q_{2}+iw\right)  ,
\]
and%
\[
\Upsilon_{1}\left(  s,p_{1},q_{1},q_{2},p_{2},\upsilon;it,iw\right)
=-\frac{2^{p_{1}+q_{1}-s}}{\left(  1-p_{1}-it\right)  _{s}}\Psi_{1}\left(
s,p_{1},q_{1},p_{1}+p_{2}+1,q_{1}+q_{2}-1,\upsilon;t,w\right)
\]
and%
\[
\Upsilon_{2}\left(  n,p_{2},q_{2},q_{1},p_{1},\upsilon;it,iw\right)
=\frac{2^{p_{2}+q_{2}-n}}{\left(  1-p_{2}-it\right)  _{n}}\Psi_{2}\left(
n,p_{2},q_{2},p_{2}+p_{1}+1,q_{2}+q_{1}-1,\upsilon;t,w\right)  .
\]

\end{theorem}

\begin{remark}
For $p_{1}=p_{2}$ and $q_{1}=q_{2}$, the weight function in Theorem 37 is positive.
\end{remark}

\begin{center}
$\mathtt{FUNDING}$
\end{center}

The author G\"{u}ldo\u{g}an Lekesiz E. in this work has been partially
supported by the Scientific and Technological Research Council of T\"{u}rkiye
(TUBITAK) (Grant number 2218-122C240).\bigskip

\begin{center}
DECLARATIONS
\end{center}

\textbf{$C$}$\mathbf{onflict\ of\ Interest\ }$The authors declared that they
have no conflict of interest.

\end{document}